\documentclass[11pt]{article}
\usepackage{setspace}
\usepackage{cprotect}
\usepackage{amsmath,amssymb}
\usepackage{amsthm}
\usepackage{url}
\usepackage{fullpage}
\usepackage{makeidx}
\usepackage{enumerate}
\usepackage{fullpage}
\usepackage{graphicx,float,psfrag,epsfig,caption,subcaption}
\usepackage{epstopdf}
\usepackage{amssymb}
\usepackage{mathrsfs}
\usepackage{color}
\usepackage[mathscr]{euscript}
\usepackage{bbm}
\usepackage{accents}
\usepackage[style=alphabetic,maxbibnames=99]{biblatex}
\usepackage{hyperref}

\newtheorem{thm}{Theorem}[section]
\newtheorem*{thm*}{Theorem}

\newtheorem{lem}[thm]{Lemma}

\theoremstyle{definition}
\newtheorem{defn}{Definition}[section]

\theoremstyle{remark}

\newcommand{\E}{\mathbb{E}}
\newcommand{\Z}{\mathbb{Z}}
\newcommand{\bP}{\mathbb{P}}
\newcommand{\sn}{\sqrt{n}}
\newcommand{\sm}{\sqrt{m}}

\newcommand{\given}{\;\middle|\;}

\DeclareMathOperator\erf{erf}

\title{Proof of the Contiguity Conjecture and Lognormal Limit for the Symmetric Perceptron
}
\author{Emmanuel Abbe \thanks{Institute of Mathematics, EPFL, Lausanne, CH-1015, Switzerland. Email: emmanuel.abbe@epfl.ch.} \and Shuangping Li \thanks{PACM, Princeton University, Princeton, NJ, 08544, USA. Email: sl31@princeton.edu.} \and Allan Sly \thanks{Department of Mathematics, Princeton University, Princeton, NJ, 08544, USA. Email: allansly@princeton.edu.}}
\date{}

\begin{document}
\maketitle

\begin{abstract}
We consider the symmetric binary perceptron model, a simple model of neural networks that has gathered significant attention in the statistical physics, information theory and probability theory communities, with recent connections made to the performance of learning algorithms in Baldassi et al.\ '15.

We establish that the partition function of this model, normalized by its expected value, converges to a lognormal distribution. As a consequence, this allows us to establish several conjectures for this model: (i) it proves the contiguity conjecture of Aubin et al.\ '19 between the planted and unplanted models in the satisfiable regime; (ii) it establishes the sharp threshold conjecture; (iii) it proves the frozen 1-RSB conjecture in the symmetric case, conjectured first by Krauth-M\'ezard '89 in the asymmetric case. 

In a recent work of Perkins-Xu \cite{perkins2021frozen}, the last two conjectures were also established by proving that the partition function concentrates on an exponential scale, under an analytical assumption on a real-valued function. This left open the contiguity conjecture and the lognormal limit characterization, which are established here unconditionally, with the analytical assumption verified. In particular, our proof technique relies on a dense counter-part of the  small graph conditioning method, which was developed for sparse models in the celebrated work of Robinson and Wormald.     
\end{abstract}

\section{Introduction}
The binary perceptron is a simple model used to study the structural properties of zero loss solutions in neural networks. It was introduced in the 60s by Cover\footnote{Mainly for the spherical case.} \cite{cover1965geometrical} and in the 80s in the statistical physics literature with detailed characterizations put forward by Gardner and Derrida \cite{gardner1988optimal} and Krauth and M\'ezard \cite{krauth1989storage}. More recently, the structural properties of its solution space have been related to the behavior of algorithms for learning neural networks in \cite{baldassi2016unreasonable,baldassi2016local,braunstein2006learning,baldassi2015subdominant} and several probabilistic results have been established in \cite{kim1998covering,talagrand1999intersecting,stojnic2013discrete,ding2019capacity,aubin2019storage} (see further discussions below). 

There exist several model variants and questions motivated by both memorization \cite{cover1965geometrical,gardner1988optimal,krauth1989storage} and generalization properties \cite{opper}. We focus here on the capacity (or memorization) problem\footnote{We refer to \cite{cover1965geometrical,gardner1988optimal} for relations between the capacity and the memorization capability of neural networks.}, the symmetric binary (or Ising) model  and the constrained satisfaction point of view.  

Namely, let $G$ be an $m$ by $n$ matrix with i.i.d.\ entries taking value in $\{+1,-1\}$ with equal probability. Fix a positive real number $\kappa$, and consider the following constraints:  
\begin{align*}
S_j(G):=\left\{X\in\{-1,+1\}^n: \frac{1}{\sqrt{n}}\left|\sum_{i=1}^n
	G_{j,i} X_i\right|
	\leq\kappa \right \}, \quad j=1,\cdots,m.
\end{align*}
The capacity problem is concerned by characterizing regimes of $m$ and $n$ for which solutions to the above constraints exist, i.e., for which  
\begin{align*}
    S^{m}(G):=\bigcap_{j=1}^m S_j(G)\neq \emptyset.
\end{align*}
More precisely, the capacity
 is defined by\footnote{Note that $m_{c,\kappa}(n)$ is a random variable that is a function of $G$.} 
\begin{align}
    m^*_{\kappa}(n) : = \max\{ m \ge 1 : S^m(G) \ne \emptyset\} \,.
\end{align}
In the asymmetric variant of the model, originally studied in \cite{cover1965geometrical,gardner1988optimal,krauth1989storage}, the constraints are given by $
\tilde{S}_j(G):=\left\{X\in\{-1,+1\}^n: \frac{1}{\sqrt{n}}\sum_{i=1}^n
	G_{j,i} X_i
	\geq\kappa \right \}$ for $1 \leq j\leq m$
with similar definitions for $\tilde{S}^m$ and $\tilde{m}^*$. Compared to the asymmetric binary perceptron (ABP), the symmetric binary perceptron (SBP) captures all of the conjectural behaviors of the perceptron but it is mathematically more tractable (see further discussions below).

Despite the simplicity of definitions,  several long-standing conjectures remain as well as intriguing connections to learning problems.  
\begin{itemize}
\item {\bf Capacity sharp threshold conjecture.} 
The simplest structural property of the solution space is the non-emptiness. 
In \cite{krauth1989storage}, it was conjectured that for the asymmetric binary perceptron (ABP), $m^*_{\kappa}(n)/n$ converges to an explicit constant $\alpha^*_\kappa$ as $n$ diverges,  equivalently given by $\alpha^*_\kappa = \sup \{ \alpha \ge 0: \lim \inf_{n \to \infty} \mathbb{P} (S^{\lfloor \alpha n \rfloor}(G) \ne \emptyset) =1 \}$. More specifically, it is conjectured that a sharp threshold phenomenon takes place at $\alpha^*_\kappa$, with the existence and absence of solutions taking place with high probability below and above capacity. 

This conjecture was extended to the symmetric binary perceptron (SBP) model in  \cite{aubin2019storage}, with a partial result obtained in this case, showing that solutions do not exist with high probability above the capacity, and exist with positive probability below capacity. Such a positive probability result was also recently obtained by  \cite{ding2019capacity} in the more challenging setting of ABP. However, in both cases, a second moment method yields a ratio between the second moment and first moment squared that is not $1$ but a positive constant, preventing to establish the high-probability result below capacity. 

To settle the sharp threshold conjecture for SBP, a possibility is to combine the result of Xu \cite{xu2019sharp}, which is currently written for the ABP setting, with the result of \cite{aubin2019storage} obtained for the SBP setting. However, this would not allow to settle the next conjecture.   

\item {\bf Contiguity conjecture.} Proving the existence of solutions below capacity with high probability follows from a stronger property of the solution space. Namely, that the distribution of the matrix $G$ (the random model) is contiguous to the distribution of the matrix $G$ in the planted model, i.e., in the model where the matrix $G$ is drawn conditioned on it satisfying a planted solution that is drawn uniformly at random. In particular, this contiguity conjecture is stated in \cite{aubin2019storage} for the SBP. Note that an additional consequence of contiguity is that there is no algorithm that can distinguish between the random and planted models with high probability.  

Interestingly, this characterization may appear to be non-consequential for other topological properties of the solution space, but it turns out to be crucial to establish a striking topological property - the strong freezing of the solutions below capacity - discussed next.  

\item {\bf Strong freezing conjecture.} 
This phenomenon was originally discovered by  \cite{krauth1989storage} and subsequently strengthen in  \cite{huang2013entropy} and \cite{huang2014origin} (see also \cite{zdeborova2008constraint,zdeborova2008locked,zdeborova2011quiet} for general CSPs). It lies at the core of the early statistical physics study of the perceptron model. 

In the case of sparse random CSPs, an important role is played by two topological thresholds involving clustering and freezing. The clustering property asserts that at sufficiently large constraint density, clusters of connected solutions break apart into exponentially many components separated by linear distance. This was described precisely in  \cite{krzakala2007gibbs} and proved in \cite{achlioptas2008algorithmic} by analyzing the planted model.
The second topological property is freezing, where in a typical cluster of solutions a linear number of variables are frozen, i.e., they take the same value in every solution in the cluster. In the case of coloring on sparse random graphs, the freezing threshold was established in \cite{molloy2012freezing}. Both of these thresholds occur well before the satisfiability thresholds in a range of sparse random CSPs such as $k$-SAT, $k$-NAESAT and random colourings.  Moreover, these thresholds are conjectured to characterize the limit of efficient algorithms.

Interestingly, in the perceptron models, the freezing property - called frozen 1-RSB - has been conjectured to take place at all positive density and not just `close' to capacity.  In particular,  \cite{krauth1989storage} conjectured that for ABP, typical solutions belong to clusters of vanishing entropy density (i.e., with $2^{o(n)}$ solutions at  linear distance), whereas~\cite{huang2013entropy} conjectured that such solutions are indeed isolated. The latter conjecture was extended to the SBP model in \cite{aubin2019storage} and \cite{baldassi2020clustering}, making the SBP model as interesting as the ABP one for such structural properties.  

\item {\bf Small graph conditioning for dense models.}
A possible approach to tackle the previous conjecture would be to obtain an analog of the small graph conditioning method in the dense ABP model considered here.  For sparse random constraint satisfaction problems, the small graph conditioning method has been used to establish contiguity for the planted distribution and determine the asymptotic law of the partition function.  This method is based around counting small cycles in the graph and accounting for their influence on $Z$.  The perceptron model is a naturally dense system but in Section~\ref{ss:smallgraph} we construct the right analogue of the small graph conditioning method by summing over products of matrix entries around cycles.

\item {\bf Learning algorithms} The latter conjectures have pointed to an interesting and novel phenomenon for learning algorithms on neural networks \cite{baldassi2015subdominant}. On the one hand, the perceptron model (symmetric or not) is conjectured to be `typically hard', as most solutions are completely frozen. On the other hand, efficient algorithms have been shown empirically to succeed in finding solutions, suggesting - if the freezing conjecture turned out to be correct - that such algorithms find atypical solutions, i.e., are part of rare clusters with atypical structural properties. This was shown empirically\footnote{See also \cite{kim1998covering} for  majority vote algorithms and extension of \cite{kim1998covering} to (symmetric) perceptrons with $\kappa>0$ in \cite{elmahdi}} in \cite{braunstein2006learning} with message passing algorithms. Further connections to learning algorithms are given in \cite{braunstein2006learning,baldassi2015subdominant}.

\end{itemize}

{\bf Our contribution.} This paper proves the above conjectures for the SBP model (sharp threshold, contiguity and strong freezing). Furthermore, it derives these results by obtaining an explicit characterization of the partition function, showing that the number of solutions normalized by its expectation tends in distribution to a lognormal distribution whose moments are characterized: for any $\alpha$ below an explicit threshold, 
\begin{align}
       \frac{Z(G)}{\E(Z(G))} \,\,\, \xrightarrow{d} \,\,\, \mathrm{Lognormal}\left( \mu(\kappa), \sigma^2(\kappa) \right),
\end{align}
where $\mu(\kappa),\sigma^2(\kappa)$ are given in Section \ref{ss:maintheorem}. We expect these methods to be more widely applicable to a range of dense constraint satisfaction problems.

While writing this paper, we became aware of the recent work \cite{perkins2021frozen}, which shows that the partition function of the symmetric perceptron concentrates on an exponential scale, using an inductive argument as constraints are added, and contingent on an analytical assumption about a real-valued function. This is sufficient to imply the sharp threshold and freezing conjectures (contingent to the assumption), but leaves open the contiguity conjecture and the lognormal limit characterization, which are established here by characterizing the distribution of the partition function using a counter-part of the small graph conditioning method. We further establish the analytical assumption to obtain unconditional results.

\section{Main Results}
\subsection{Main Theorem}\label{ss:maintheorem}
Define the probability
\begin{align*}
    P_\kappa=\bP(|N|\leq \kappa),
\end{align*}
where $N$ follows the standard normal distribution and define the capacity density
\begin{align*}
    \alpha_c(\kappa)=-\frac{\log 2}{\log P_\kappa}.
\end{align*}
We denote $Z(G,m,n):=|S(G)|$ as the number of solutions to the symmetric binary perceptron problem and write $Z(G)$ in short when $m$ and $n$ are clear. To state our main theorem, we define
\begin{align*}
    \mu_{2,\kappa}=\frac{1}{P_\kappa}\int_{[-\kappa,\kappa]} \frac{1}{\sqrt{2\pi}} e^{-\frac{x^2}{2}} x^2 dx \quad \text{and} \quad \beta=-\frac{\sqrt{\alpha}}{2}(1-\mu_{2,\kappa}).
\end{align*}
For notation purposes, we use $\xrightarrow{d}$ to denote converge in distribution and $\xrightarrow{p}$ to denote converge in probability.
\begin{thm}\label{t:lognormal}
Let $\kappa>0$ and $0<\alpha<\alpha_c(\kappa)$. Take $m=\lfloor\alpha n\rfloor$. Then,
\begin{align*}
    \frac{Z(G)}{\E(Z(G))} \xrightarrow{d} \mathrm{Lognormal}\left(\frac{1}{4}\log(1-4\beta^2)+\beta^2,-\frac{1}{2}\log(1-4\beta^2)-2\beta^2\right),
\end{align*}
as $n \rightarrow \infty$.
\end{thm}
This theorem describes the limiting distribution of $Z(G)/\E [Z]$. This enables us to study other properties of the model. Note that when $\alpha<\alpha_c(\kappa)$, we have $-1/2<\beta<0$. (This is a direct consequence of Lemma \ref{h:1} and will be proven in Section \ref{ss:hypo}). Therefore the parameters on the right hand side is well-defined.

\subsection{Consequences}
A few immediate consequences are presented below. We use $\bP$ to denote the law of $G$ when entries of $G$ take i.i.d. value in $\{\pm 1\}$ with probability $1/2$.
\begin{defn}[Planted Model]\label{d:planted}
We use $\bP^*$ to denote the law of $G$ condition on $X\in S(G)$, where $X$ is a solution uniformly chosen at random from $\{\pm 1\}^n$.
\end{defn}

\begin{thm}[Contiguity]\label{t:contiguity}
Let $\kappa>0$ and $0<\alpha<\alpha_c(\kappa)$. Take $m=\lfloor\alpha n\rfloor$. Then $\bP$ and $\bP^*$ are mutually contiguous, i.e. for any sequence of events $A_n$, $\bP(A_n)\rightarrow 0$ if and only if $ \bP^*(A_n)\rightarrow 0$.
\end{thm}

\begin{thm}[Sharp Threshold]\label{t:threshold}
For any $\kappa>0$ and $\alpha>\alpha_c(\kappa)$,
\begin{align*}
    \lim_{n\rightarrow \infty}\bP(Z(G,\lfloor\alpha n\rfloor,n)\geq 1)=0.
\end{align*}
For any $0<\alpha<\alpha_c(\kappa)$,
\begin{align*}
    \lim_{n\rightarrow \infty}\bP(Z(G,\lfloor\alpha n\rfloor,n)\geq 1)=1.
\end{align*}
\end{thm}

\begin{thm}[Freezing of typical solutions]\label{t:freezing}
Let $\kappa>0$ and $0<\alpha<\alpha_c(\kappa)$. Take $m=\lfloor\alpha n\rfloor$. There exists a constant $\delta>0$ such that
\begin{align*}
    \lim_{n\rightarrow \infty}\bP(\{X_2\in S(G):d(X,X_2)\leq \delta n\}=\{X\})=1,
\end{align*}
where $X$ is chosen uniformly at random from $S(G)$. Here the distance function $d(X,X_2)=(n-\langle X,X_2\rangle)/2$ is the Hamming distance.
\end{thm}
This means that for a typical solution $X$, there is no other solutions of distance less than or equal to $\delta n$ from it. Or in other words, for a typical solution $X$, if we flip a small constant proportion of its entries, the resulting vector is not a solution.

\subsection{Small graph conditioning for dense models}\label{ss:smallgraph}

The recurring challenge in applying the second moment method is when $\E[Z^2]/(\E[Z])^2\to C > 1$ which arises in a multitude of examples in the study of combinatorial properties of random graphs and random constraint satisfaction problems.  A key breakthrough was made in work of Robinson and Wormald who developed what is known as the small graph conditioning method to prove that random regular graphs have Hamilton cycles with high probability~\cite{robinson1994almost}.  Work of Janson~\cite{janson1995random} later formalized the method and explicitly showed that it establishes contiguity of the original and planted distributions.

The method works by ``explaining'' the additional variance through the effect that the small cycles in the graph have on the partition function.  In particular, typically each small cycle has a constant multiplicative effect on the expected  partition function since the edge completing the cycle is connecting two correlated variables and differs from an edge joining two distant variables.  The cumulative contribution can be expressed as
\[
\frac{Z}{\exp(Y)\E(Z)}  \xrightarrow{p} 1,
\]
as $n\to\infty$ where
\begin{equation}\label{eq:sparseSGC}
Y=C+ \sum_{k} \delta_k C_k,
\end{equation}
with $C_k$ the number of cycles of length $k$ and $\delta_k$ model dependent parameters.  While sometimes technically challenging to implement, the small graph method has played an important role in the analysis of many models~\cite{galanis2016inapproximability,coja2016chromatic,mossel2009hardness,mossel2015reconstruction,MR3783558}.

Unlike sparse random CSPs, the perceptron models are dense with direct interactions between each pair of variables in all of the constraints.  Nonetheless, we find the  analogue of equation~\eqref{eq:sparseSGC} defining
\begin{align*}
    Y:=\sum_{k=2}^{\lfloor \log n \rfloor} \frac{2(2\beta)^kC_{k}-(2\beta)^{2k}}{4k},
\end{align*}
where the cycle counts are now generalized as
\begin{equation}\label{eq:denseSGC}
    C_{k}=\left(\frac{1}{\sn}\right)^k\left(\frac{1}{\sm}\right)^k\sum_{\substack{i_1,i_2,\cdots, i_k\in [n] \text{ disjoint}\\j_1,j_2,\cdots, j_k\in [m] \text{ disjoint}}}\prod_{\ell=1}^k G_{j_\ell,i_\ell} G_{j_{\ell},i_{\ell+1}},
\end{equation}
where we identify $i_{k+1}$ with $i_1$. Previously, in the dense case, \cite{banerjee2018contiguity} has constructed a similar term that works for dense stochastic block models. In place of counting over cycles of the graph we instead sum over the products of matrix entries of $G$ along cycles in the entries.  We prove asymptotic joint normality for the $C_k$ by computing their joint moments and showing convergence to the moments of independent Gaussians.  With this definition we show the following theorem.

\begin{thm}\label{t:ConvInP}
Let $\kappa>0$ and $0<\alpha<\alpha_c(\kappa)$. Take $m=\lfloor\alpha n\rfloor$. Then,
\begin{align*}
    \frac{Z(G)}{\exp(Y)\E(Z(G))} \xrightarrow{p} 1,
\end{align*}
as $n \rightarrow \infty$. 
\end{thm}

We now discuss the motivation behind the definition of $Y$. We want to understand the effect on the distribution of a random solution $X \in S(G)$ from the constraint matrix $G$. In the symmetric perceptron model, the marginal distribution of each entry $X_i$ is always balanced by symmetry for any $G$. The next simplest interaction is to consider the effect of the constraints on the correlation between pairs of $X_i$. Each constraint has the effect of creating a small correlation which can be either positive or negative and is of order $1/n$.
More specifically, if $G_j$ denotes the $j$-th row of $G$ and $x_p,x_q\in \{\pm 1\}$, then
\begin{align*}
    \bP(X_p=x_p,X_q=x_q\mid X\in S_j(G),G_j )&\approx\frac{1}{4}\left(1-\frac{(1-\mu_{2,\kappa})x_p x_q G_{j,p} G_{j,q}}{n}\right)\\
    &\propto \exp\left(-\frac{(1-\mu_{2,\kappa})x_p x_q G_{j,p} G_{j,q}}{n}\right).
\end{align*}
If the effect of the rows were multiplicative then we might expect that
\begin{align*}
    \bP(X_p=x_p,X_q=x_q \mid X\in S(G),G)\propto \exp\left(- \frac{\sqrt{\alpha}(1-\mu_{2,\kappa}) J_{p,q}x_p x_q}{\sn}\right).
\end{align*}
where
\[
J_{p,q}= \frac{ \sum_{j=1}^m G_{j,p} G_{j,q}}{\sm}.
\]
Since the rows are independent, asymptotically $J_{p,q}\sim\mathcal{N}(0,1)$. This is strikingly similar to the interaction terms in the famous Sherrington-Kirkpatrick model. In the SK model, an expansion of the partition function can be expressed as a sum of products of weights around cycles as was shown in~\cite{aizenman1987some,banerjee2020fluctuation}. While the partition function for the symmetric perceptron does not have an analogous expansion, we can still write $Y$ as a similar sum over weighted cycles.

\section{Proof Outline}
A very important step in our proof is to establish Theorem \ref{t:ConvInP}, which states for any $\kappa>0$, $0<\alpha<\alpha_c(\kappa)$ and $m=\lfloor\alpha n\rfloor$, we have
\begin{align*}
    \frac{Z(G)}{\exp(Y)\E(Z(G))} \xrightarrow{p} 1,
\end{align*}
as $n \rightarrow \infty$. 
In other words, the number of solutions normalized by its expectation, $Z(G)/\E[Z]$, is close in probability to $\exp(Y)$. Theorem \ref{t:ConvInP} helps us to have a good description of $Z(G)$. In particular, by studying the distribution of $Y$, we can characterize the limiting distribution of the number of solutions (Theorem \ref{t:lognormal}). Indeed, we will show that $Y$ is asymptotically normal and, therefore, $\exp(Y)$ is asymptotically lognormal. 

Moreover, the asymptotic limit directly implies the existence of solutions below the threshold. Together with a first moment computation that gives us the non-existence of solutions above the threshold, we are able to establish the sharp threshold results (Theorem \ref{t:threshold}). 

Further, note that the planted and the original model are connected in the following way. The Radon–Nikodym derivative satisfies
\begin{align*}
    \frac{d\bP^*}{d\bP}=\frac{Z(G)}{\E(Z(G))}.
\end{align*}
Here recall that we use $\bP^*$ to denote the law of $G$ for the planted model and $\bP$ for the original model. As we have characterized the distribution of the right hand side, the contiguity property follows.

To establish freezing of typical solutions (Theorem \ref{t:freezing}), we apply our contiguity theorem. Contiguity is a strong property that connects the planted model to the original model, which allows us to say that any event that happens with high probability in one side, also happens with high probability on the other side. In particular, in our case, the freezing property can be easily shown in the planted model. With contiguity, we can translate the result to the original model and therefore prove the theorem.

In order to show Theorem \ref{t:ConvInP}, we consider a truncated $Y$. For a positive integer $M_1\geq 2$, define
\begin{align*}
    Y_{M_1}:=\sum_{k=2}^{M_1} \frac{2(2\beta)^kC_{k}-(2\beta)^{2k}}{4k}.
\end{align*}
\begin{lem}\label{l:secondmoment}
Let $\kappa>0$ and $0<\alpha<\alpha_c(\kappa)$. Take $m=\lfloor\alpha n\rfloor$. For any $\varepsilon>0$, there exist integers $M_1\geq 2$ and $M_2>0$ such that for $n$ large,
\begin{align}
    \E \left (\frac{Z(G)}{\exp(Y_{M_1}\mathbbm{1}[|Y_{M_1}|\leq M_2])} \right )\geq (1-\varepsilon) \E (Z(G)  ),\label{eq:firstm}
\end{align}
and
\begin{align}
    \E \left (\frac{Z(G)^2}{\exp(2Y_{M_1}\mathbbm{1}[|Y_{M_1}|\leq M_2])} \right )\leq (1+\varepsilon) \E (Z(G))^2. \label{eq:secondm}
\end{align}
\end{lem}
This lemma enables us to say $Z(G)/\E [Z]$ is close to $\exp(Y_{M_1}\mathbbm{1}[|Y_{M_1}|\leq M_2])$, which is also close to $\exp(Y)$. This will imply Theorem \ref{t:ConvInP}. 
 
In order to prove Lemma \ref{l:secondmoment} and to study the distribution of $Y$, we need to learn the joint distribution of the cycle counts. Recall the definition of the planted model in Definition \ref{d:planted}. We consider another definition when two solutions are planted.
\begin{defn}[Two solutions planted]
We sample two solutions $X_1$ and $X_2$ with replacement uniformly at random from $\{\pm 1\}^n$ condition on $\langle X_1,X_2 \rangle=t\sn$. We use $\bP^{*2}_t$ to denote the distribution of $G$ condition on $X_1,X_2\in S(G)$. 
\end{defn}
With this definition, we have the following characterization of $C_k$. 
\begin{lem}\label{l:normal}
Let $\kappa>0$ and $0<\alpha<\alpha_c(\kappa)$. Take $m= \lfloor\alpha n\rfloor$. Under $\bP$, for any integer $k\geq 2$,
\begin{align*}
    \left(\frac{C_{2}}{2},\cdots, \frac{C_{k}}{\sqrt{2k}}\right) \xrightarrow{d} \mathcal{N}(\mathbf{0},I_{k-1}),
\end{align*}
as $n$ goes to infinity. Under $\bP^*$, for any integer $k\geq 2$,
\begin{align*}
    \left(\frac{C_{2}-(2\beta)^2}{2},\cdots, \frac{C_{k}-(2\beta)^k}{\sqrt{2k}}\right) \xrightarrow{d} \mathcal{N}(\mathbf{0},I_{k-1}),
\end{align*}
as $n$ goes to infinity. For any $|t|\leq \log n$ and any integer $k\geq 2$, under $\bP^{*2}_t$, 
\begin{align*}
    \left(\frac{C_{2}-2(2\beta)^2}{2},\cdots, \frac{C_{k}-2(2\beta)^k}{\sqrt{2k}}\right) \xrightarrow{d} \mathcal{N}(\mathbf{0},I_{k-1}),
\end{align*}
as $n$ goes to infinity.
\end{lem}
Note that the purpose of studying the joint distribution of cycles under $\bP$, $\bP^*$ and $\bP^{*2}_t$ are different. The joint distribution of cycles under $\bP$ will allow us to characterize the distribution of $Y$, since $Y$ is a weighted sum of the cycles; whereas distributions under $\bP^*$ and $\bP^{*2}_t$ will be useful in computing a weighted first and second moment of $Z(G)$, see the left hand side of $\eqref{eq:firstm}$ and~$\eqref{eq:secondm}$ respectively. With this lemma, we have the following results on the distribution of $Y$. Define a function
\begin{align*}
    L(M_1):=\sum_{k=2}^{M_1} \frac{(2\beta)^{2k}}{k}.
\end{align*}

\begin{lem}\label{l:YMconverge}
Let $\kappa>0$ and $0<\alpha<\alpha_c(\kappa)$. Take $m= \lfloor\alpha n\rfloor$. Under $\bP$, for any integer $M_1\geq 2$,
\begin{align*}
    Y_{M_1}\xrightarrow{d} \mathcal{N}\left(-\frac{1}{4}L(M_1),\frac{1}{2}L(M_1)\right),
\end{align*}
as $n$ goes to infinity. Under $\bP^*$, for any integer $M_1\geq 2$,
\begin{align*}
    Y_{M_1}\xrightarrow{d} \mathcal{N}\left(\frac{1}{4}L(M_1),\frac{1}{2}L(M_1)\right),
\end{align*}
as $n$ goes to infinity. For any $|t|\leq \log n$ and any integer $M_1\geq 2$, under $\bP^{*2}_t$, 
\begin{align*}
     Y_{M_1}\xrightarrow{d} \mathcal{N}\left(\frac{3}{4}L(M_1),\frac{1}{2}L(M_1)\right),
\end{align*}
as $n$ goes to infinity.
\end{lem}
We note that $L(M_1)=\sum_{k=2}^{M_1} \frac{(2\beta)^{2k}}{k}$ is the Taylor series expansion for $-\log(1-4\beta^2)-4\beta^2$. Therefore, we have $L(M_1)\to -\log(1-4\beta^2)-4\beta^2$ as $M_1 \to \infty$. This explains the parameters for the lognormal distribution in Theorem \ref{t:lognormal}.

We prove Lemma \ref{l:normal} by computing joint moments of $C_k$ and showing convergence to the moments of independent Gaussians. To this end, we expand the sums and reorganize them according to underlying hypergraphs that they represent. With some combinatorial computations, we are able to show that the leading order terms converge to the corresponding moments of independent Gaussians. 

The second moment computation in the proof of Lemma \ref{l:secondmoment} requires a property of a special function. This property appears as a hypothesis in \cite{aubin2019storage,perkins2021frozen}. It is an essential condition to establish that most pairs of solutions are of distance around $n/2$ from each other. More specifically, for $0\leq x\leq 1$, let $N_1$ and $N_2$ be two standard normal random variables with correlation $2x-1$. Define
\begin{align*}
    q_\kappa(x):=\bP(|N_1|\leq \kappa, |N_2|\leq \kappa) \quad \text{and} \quad F_{\kappa,\alpha}(x):=\alpha \log q_\kappa(x)-x \log x-(1-x)\log(1-x).
\end{align*}
\begin{lem}\label{h:1}
For any $\kappa>0$ and $0<\alpha<\alpha_c(\kappa)$, $F_{\kappa,\alpha}(x)<F_{\kappa,\alpha}(\frac{1}{2})$, for any $x\in [0,1]$ except at $1/2$.
\end{lem}
The proof is given in Section \ref{ss:hypo}.

\section{Proofs of Main Results}
We start by defining a few discrete analogues of the definitions in Section \ref{ss:maintheorem}. Define \begin{align*}
    P_{\kappa,n}=\bP(|X|\leq \kappa \sn),
\end{align*}
where $X$ is the sum of $n$ independent Rademacher random variables. And define
\begin{align*}
    \mu_{2,\kappa,n}=\frac{1}{2^n n P_{\kappa,n}} \sum_{t=0}^n  \binom{n}{t} (2t-n)^2\mathbbm{1}(|2t-n|\leq \kappa \sn) \quad \text{and} \quad \beta_{n}=-\frac{\sqrt{m}}{2\sqrt{n}}(1-\mu_{2,\kappa,n}).
\end{align*}
Note that when $m=\lfloor \alpha n \rfloor$, $P_{\kappa,n}$, $\mu_{2,\kappa,n}$ and $\beta_{n}$ converge to $P_{\kappa}$, $\mu_{2,\kappa}$ and $\beta$ as $n \to \infty$.

\subsection{Proof of Lemma \ref{l:normal}}\label{ss:normal}
To prove the asymptotic normality of $C_{k}$, we define a shifted cycle count. Without loss of generality, when we plant one solution, we assume that the all one vector is chosen. When we plant two solutions, we assume that one of them is the all one vector and the other has $1$'s in the first $n/2+t\sqrt{n}/2$ entries and $-1$'s in the rest. When we are working with $\bP^*$, define
\begin{align*}
    \bar C_k(G)=\left(\frac{1}{\sn}\right)^k\left(\frac{1}{\sm}\right)^k\sum_{\substack{i_1,i_2,\cdots, i_k\in [n] \text{ disjoint}\\j_1,j_2,\cdots, j_k\in [m] \text{ disjoint}}}\prod_{\ell=1}^k \left(G_{j_\ell,i_\ell} G_{j_\ell,i_{\ell+1}}-\frac{2\beta_n}{\sqrt{mn}}\right),
\end{align*}
Similarly, for $\bP^{*2}_t$, we define
\begin{align*}
    \bar C_k(G)=\left(\frac{1}{\sn}\right)^k\left(\frac{1}{\sm}\right)^k\sum_{\substack{i_1,i_2,\cdots, i_k\in [n] \text{ disjoint}\\j_1,j_2,\cdots, j_k\in [m] \text{ disjoint}}}\prod_{\ell=1}^k \left(G_{j_\ell,i_\ell} G_{j_\ell,i_{\ell+1}}-\frac{4\beta_n\mathbbm{1}(i_\ell,i_{\ell+1}\in Q)}{\sqrt{mn}}\right),
\end{align*}
where $i_\ell,i_{\ell+1}\in Q$ whenever $i_\ell,i_{\ell+1}\in [n/2+t\sqrt{n}/2]$ or $i_\ell,i_{\ell+1}\in [n]\backslash[n/2+t\sqrt{n}/2]$. Here we use $[k]$ to denote the set $\{1,2,\cdots,k\}$.

\begin{lem}\label{l:cycle}
Let $\kappa>0$ and $0<\alpha<\alpha_c(\kappa)$. Take $m= \lfloor\alpha n\rfloor$. Under $\bP^*$ or $\bP^{*2}_t$ where $|t|\leq \log n$, for any integer $k\geq 2$,
\begin{align*}
    \left(\frac{\bar C_2}{2},\cdots, \frac{\bar C_k}{\sqrt{2k}}\right) \xrightarrow{d} \mathcal{N}(\mathbf{0},I_{k-1}),
\end{align*}
as $n$ goes to infinity. Under $\bP$, for any integer $k\geq 2$,
\begin{align*}
    \left(\frac{ C_2}{2},\cdots, \frac{ C_k}{\sqrt{2k}}\right) \xrightarrow{d} \mathcal{N}(\mathbf{0},I_{k-1}),
\end{align*}
as $n$ goes to infinity. 
\end{lem}

Moreover, we have that
\begin{lem}\label{l:cycleRterm}
Let $\kappa>0$ and $0<\alpha<\alpha_c(\kappa)$. Take $m= \lfloor\alpha n\rfloor$. Under $\bP^*$, for any integer $k\geq 2$,
\begin{align*}
    C_k - \bar  C_k- (2\beta_n)^k \xrightarrow{p} 0,
\end{align*}
as $n$ goes to infinity. For any $|t|\leq \log n$ and any integer $k\geq 2$, under $\bP^{*2}_t$, 
\begin{align*}
    C_k - \bar  C_k- 2(2\beta_n)^k \xrightarrow{p} 0,
\end{align*}
as $n$ goes to infinity.
\end{lem}
The two lemmas together imply Lemma \ref{l:normal}.

Fix an integer $1\leq j\leq m$. Consider a multigraph $H_j=([n],E(H_j))$. We write $V(H_j)$ to be the set of all non-isolated vertices. For each edge $e=(p,q)\in E(H_j)$, we define
\begin{align*}
    G_e:=G_{j,p} G_{j,q} \quad \text{and } \quad  G_{H_j}:=\prod\limits_{e\in E(H_j)} G_e.
\end{align*}
We also define
\begin{align*}
    \bar G_e:=G_{j,p} G_{j,q}-\frac{2\beta_n}{\sqrt{mn}} \quad \text{and } \quad \bar G_{H_j}:=\prod\limits_{e\in E(H_j)} \bar G_e,
\end{align*}
when we work with $\E^*$. And define 
\begin{align*}
    \bar G_e:=G_{j,p} G_{j,q}-\frac{4\beta_n\mathbbm{1}(p,q\in Q)}{\sqrt{mn}} \quad \text{and } \quad \bar G_{H_j}:=\prod\limits_{e\in E(H_j)}\bar G_e,
\end{align*}
when we work with $\E^{*2}_t$.

\begin{lem}\label{l:rows}
If $H_j$ is an even graph, then $\E^*(\bar G_{H_j})=\mathcal{O}(1)$. If there are exactly two vertices in $V(H_j)$ with odd degrees and $|E(H_j)|\geq 2$, then $\E^*(\bar G_{H_j})=\mathcal{O}(1/n)$. Else, $\E^*(\bar G_{H_j})=\mathcal{O}(1/n^2)$.
\end{lem}
\begin{proof}
For simplicity of notation, we drop $j$ in the subscript. Note that we assumed that the all one vector is the planted solution. So $\E^*$ is symmetric for any entries $G_i$. In particular, we have that for disjoint $i_1,i_2,\cdots, i_k\in [n]$,
\begin{align*}
    \E^* (G_{i_1}\cdots G_{i_k})&=\E^* (G_1\cdots G_k)=\E\left(G_1\cdots G_k \given |\sum_{i=1}^n G_i |\leq \kappa \sn\right).
\end{align*}
The expected product is zero when $k$ is odd by symmetry. When $k$ is even, we notice that by Stirling's approximation,
\begin{align*}
    &\bP^*(G_1=1,\cdots ,G_a=1, G_{a+1}=-1,\cdots, G_k=-1)\\
    &=\frac{1}{P_{\kappa,n}}2^{-n}\sum\limits_{\substack{t:|t|\leq \kappa\sn\\ \frac{n}{2}+\frac{t\sn }{2} \in \Z  }}\binom{n-k}{\frac{n}{2}+\frac{t\sn}{2}-a}\\
    &=\frac{1}{P_{\kappa,n}}2^{-n}\sum\limits_{\substack{t:|t|\leq \kappa\sn\\ \frac{n}{2}+\frac{t\sn }{2} \in \Z  }}\binom{n}{\frac{n}{2}+\frac{t\sn}{2}}2^{-k}(1+\frac{(2a-k)t}{\sn}+\frac{((k-2a)^2-k)(t^2-1)}{2n}+\frac{\mathrm{poly}_{a,k,t} }{6n^{3/2}}+\mathcal{O}(\frac{1}{n^2}))\\
    &=2^{-k}(1+\frac{((k-2a)^2-k)(\mu_{2,\kappa,n}-1)}{2n}+\mathcal{O}(\frac{1}{n^2})),
\end{align*}
where $\mathrm{poly}_{a,k,t}=(4 a + 8 a^3 - 2 k - 6 a k - 12 a^2 k + 3 k^2 + 6 a k^2 -k^3)t(t^2-3)$. This further implies that
\begin{align}
    \E^*(G_1\cdots G_k)&=\sum_{a=0}^k (-1)^{k-a} \binom{k}{a} 2^{-k}(1+\frac{((k-2a)^2-k)(\mu_{2,\kappa,n}-1)}{2n}+\mathcal{O}(\frac{1}{n^2}))\\
    &=\begin{cases}
    \mathcal{O}(\dfrac{1}{n^2}), &\text{ if } k\geq 3\\
    \dfrac{2\beta_n}{\sqrt{mn}}+\mathcal{O}(\dfrac{1}{n^2}), &\text{ if } k=2.
    \end{cases}\label{eq:Gexpect}
\end{align}
Here we used the definition that $\beta_n=\sqrt{m}(\mu_{2,\kappa,n}-1)/2\sqrt{n}$. Now, if $H$ consists of a single edge, then
\begin{align*}
    \E^*(\bar G_H)&=\E^*(G_{j,p} G_{j,q}-\frac{2\beta_n}{\sqrt{mn}})\\
    &=\mathcal{O}(\frac{1}{n^2}).
\end{align*}
Otherwise,
\begin{align*}
    \E^*(\bar G_H)&=\E^*(\prod\limits_{e\in E(H)}(G_e-\frac{2\beta_n}{\sqrt{mn}}))\\
    &=\E^*(\prod\limits_{e\in E(H)} G_e )+\sum_{A\subsetneq E(H)}\E^* (\prod\limits_{e\in A} G_e\prod\limits_{e\in E(H)\backslash A}-\frac{2\beta_n}{\sqrt{mn}})\\
    &:=S_1+ S_2.
\end{align*}
Notice that when $|E(H)|\geq 2$ and $H$ has at least $4$ vertices with odd degrees, we have $S_1=\mathcal{O}(1/n^2)$ by equation \eqref{eq:Gexpect}. For the second term $S_2$, if $A$ consists of all but one edge, then $\E^* (\prod_{e\in A} G_e)=\mathcal{O}(1/n)$, and thus $S_2=\mathcal{O}(1/n^2)$. Otherwise, we directly have $S_2=\mathcal{O}(1/n^2)$. Now if $H$ has only $2$ vertices with odd degrees, then by equation \eqref{eq:Gexpect} again, $S_1=\mathcal{O}(1/n)$. As $S_2=\mathcal{O}(1/n)$ as well, $\E^*(\bar G_H)=\mathcal{O}(1/n)$. Finally, if $H$ is even, then any product in the computation is $\mathcal{O}(1)$, so the statement follows.
\end{proof}

\begin{lem}\label{l:rows2}
Let $|t|\leq \log n$. If $H_j$ is an even graph, then $\E^{*2}_t(\bar G_{H_j})=\mathcal{O}(1)$. If there are exactly two vertices in $V(H_j)$ with odd degrees and $|E(H_j)|\geq 2$, then $\E^{*2}_t(\bar G_{H_j})=\mathcal{O}(1/n)$. Else, $\E^{*2}_t(\bar G_{H_j})=\mathcal{O}(t^4/n^2)$.
\end{lem}

\begin{proof}
For simplicity, write $T=n/2+t\sn/2$ and $[T]^c=[n]\backslash [T]$. As we assumed that the two planted solutions are the all one vector and the vector with ones in the front $T$ entries, we have that $\E^{*2}_t$ is symmetric for any entries $G_i$ in $[T]$ and $[T]^c$. In particular, we have that for disjoint $i_1,i_2,\cdots, i_{k_1}\in [T]$ and $i_{k_1+1},i_{k_1+2},\cdots, i_{k_1+k_2}\in [T]^c$,
\begin{align}\label{eq:productG2}
    &\E^{*2}_t (G_{i_1}\cdots G_{i_{k_1+k_2}})=\E^{*2}_t (G_1\cdots G_{k_1}G_{T+1}\cdots G_{T+k_2})\\
    &=\E\left(G_1\cdots G_{k_1}G_{T+1}\cdots G_{T+k_2} \given |\sum_{i\in [n]} G_i|\leq \kappa \sn, |\sum_{i\in [T]} G_i-\sum_{i\in [T]^c} G_i|\leq \kappa \sn\right )
\end{align}
We define
\begin{align*}
    P_\kappa^2(t):&=\bP\left(|\sum_{i\in [n]} G_i|\leq \kappa \sn, |\sum_{i\in [T]} G_i-\sum_{i\in [T]^c} G_i|\leq \kappa \sn\right )\\
    &= 2^{-n}\sum_{x_1,x_2} \binom{\frac{n}{2}+\frac{t\sn}{2}}{\frac{n}{4}+\frac{(t+x_1+x_2)\sn}{4}}\binom{\frac{n}{2}-\frac{t\sn}{2}}{\frac{n}{4}+\frac{(-t+x_1-x_2)\sn}{4}}\\
    &=P_\kappa^2 (1+\frac{(t^2-1)(1-[x^2]_\kappa)^2}{2n}+\mathcal{O}(\frac{t^4}{n^2})).
\end{align*}
where the summation is over $x_1$ and $x_2$ that satisfy $|x_1|, |x_2|\leq \kappa$, $\frac{n}{2}+\frac{x_1\sn }{2} \in \Z $, $\frac{n}{2}+\frac{x_2\sn }{2} \in \Z$ and $n+\frac{(x_1+x_2)\sn }{2}$ even. Write $k=k_1+k_2$. The expected product in \eqref{eq:productG2} is zero when $k$ is odd by symmetry. We write $G_{a:b}=1$ to denote that $G_a=\cdots=G_b=1$. When $k$ is even, we notice that by Stirling's approximation,
\begin{align*}
    &\bP^{*2}_t(G_{1:a}=1, G_{a+1:k_1}=-1,G_{T+1:T+b}=1, G_{T+b+1:T+k_2}=-1)\\
    &=\frac{1}{P_\kappa^2(t)}2^{-n}\sum_{x_1,x_2} \binom{\frac{n}{2}+\frac{t\sn}{2}-k_1}{\frac{n}{4}+\frac{(t+x_1+x_2)\sn}{4}-a}\binom{\frac{n}{2}-\frac{t\sn}{2}-k_2}{\frac{n}{4}+\frac{(-t+x_1-x_2)\sn}{4}-b}.\\
    &=\frac{1}{P_\kappa^2(t)}2^{-2n}\sum_{x_1,x_2}2^{-k_1-k_2}2\binom{n}{\frac{n}{2}+\frac{x_1\sn}{2}}\binom{n}{\frac{n}{2}+\frac{x_2\sn}{2}}(1+\frac{\mathrm{poly}_{a,b,x_1,x_2,t}}{2n}+\mathcal{O}(\frac{t^4}{n^2})),
\end{align*}
where the summation is over $x_1$ and $x_2$ that satisfy $|x_1|, |x_2|\leq \kappa$, $\frac{n}{2}+\frac{x_1\sn }{2} \in \Z $, $\frac{n}{2}+\frac{x_2\sn }{2} \in \Z$ and $n+\frac{(x_1+x_2)\sn }{2}$ even. Here $\mathrm{Poly}_{a,b,x_1,x_2,t}=(-k+(k_1-k_2-2a+2b)^2)(x_1^2-1)+(-k+(k_1+k_2-2a-2b)^2)(x_2^2-1)+(t^2-1)(x_1^2-1)(x_2^2-1)$. Notice that all odd polynomials of $x_1$ and $x_2$ cancel in the sum so we omit them here. Note that
\begin{align*}
    \frac{1}{P_\kappa^2(t)}2^{-2n}\sum_{x_1,x_2}2^{-k_1-k_2}2\binom{n}{\frac{n}{2}+\frac{x_1\sn}{2}}\binom{n}{\frac{n}{2}+\frac{x_2\sn}{2}}x_1^2=\frac{P_{\kappa,n}^2}{P_\kappa^2(t)}\mu_{2,\kappa,n}.
\end{align*}
Then we have that
\begin{align*}
    &\bP^{*2}_t(G_{1:a}=1, G_{a+1:k_1}=-1,G_{T+1:T+b}=1, G_{T+b+1:T+k_2}=-1)\\
    &=\frac{P_\kappa^2}{P_\kappa^2(t)}2^{-k}(1+\frac{2((k_1-2a)^2+(k_2-2b)^2-k))(\mu_{2,\kappa,n}-1)+(t^2-1)(1-[x^2]_\kappa)^2}{2n}+\mathcal{O}(\frac{t^4}{n^2}))\\
    &=2^{-k}(1+\frac{2((k_1-2a)^2+(k_2-2b)^2-k))(\mu_{2,\kappa,n}-1)}{2n}+\mathcal{O}(\frac{t^4}{n^2})).
\end{align*}
This further implies that
\begin{align}
    &\E^{*2}_t(G_1\cdots G_{k_1}G_{T+1}\cdots G_{T+k_2})\\
    &=\sum_{a=0}^{k_1}\sum_{b=0}^{k_2} (-1)^{k-a-b} \binom{k_1}{a}\binom{k_2}{b} 2^{-k}(1+\frac{2((k_1-2a)^2+(k_2-2b)^2-k))(\mu_{2,\kappa,n}-1)}{2n}+\mathcal{O}(\frac{t^4}{n^2}))\\
    &=\begin{cases}
    \dfrac{4\beta_n}{\sqrt{mn}}+\mathcal{O}(\dfrac{t^4}{n^2}), &\text{ if } k=2, a=0 \text{ or } b=0\\
    \mathcal{O}(\dfrac{t^4}{n^2}), &\text{ else. }\\
    \end{cases}\label{eq:Gexpect2}
\end{align}
Here we used the definition that $\beta_n=\sqrt{m}(\mu_{2,\kappa,n}-1)/2\sqrt{n}$. The rest of the argument are similar. If $H$ consists of a single edge, then
\begin{align*}
    \E^{*2}_t(\bar G_H)&=\E^{*2}_t(G_{j,p} G_{j,q}-\frac{4\beta_n\mathbbm{1}(p,q\in Q)}{\sqrt{mn}})\\
    &=\mathcal{O}(\frac{t^4}{n^2}).
\end{align*}
Otherwise,
\begin{align*}
    \E^{*2}_t(\bar G_H)&=\E^{*2}_t(\prod\limits_{e\in E(H)}(G_e-\frac{4\beta_n\mathbbm{1}(e\in Q)}{\sqrt{mn}}))\\
    &=\E^{*2}_t(\prod\limits_{e\in E(H)} G_e)+\sum_{A\subsetneq E(H)}\E^{*2}_t (\prod\limits_{e\in A} G_e\prod\limits_{e\in E(H)\backslash A}-\frac{4\beta_n\mathbbm{1}(e\in Q)}{\sqrt{mn}})\\
    &:=S_1+ S_2.
\end{align*}
Notice that when $|E(H)|\geq 2$ and $H$ has at least $4$ vertices with odd degrees, we have $S_1=\mathcal{O}(t^4/n^2)$ by equation \eqref{eq:Gexpect2}. For the second term $S_2$, if $A$ consists of all but one edge, then $\E^{*2}_t (\prod_{e\in A} G_e)=\mathcal{O}(1/n)$, and thus $S_2=\mathcal{O}(1/n^2)$. Otherwise, we directly have $S_2=\mathcal{O}(1/n^2)$. Now if $H$ has only $2$ vertices with odd degrees, then by equation \eqref{eq:Gexpect2} again, $S_1=\mathcal{O}(1/n)$. As $S_2=\mathcal{O}(1/n)$ as well, $\E^{*2}_t(\bar G_H)=\mathcal{O}(1/n)$. Finally, if $H$ is even, then any product in the computation is $\mathcal{O}(1)$, so the statement follows.
\end{proof}

\begin{proof}[Proof of Lemma \ref{l:cycle}]
We establish this through the method of moments and Wick's probability theorem. We will show that, for non-negative integers $\alpha_2,\cdots,\alpha_k$,
\begin{align}
    \lim\limits_{n\rightarrow \infty}\E^{*} \bar C_2^{\alpha_2} \bar C_3^{\alpha_3} \cdots \bar C_k^{\alpha_k}=
\begin{cases}
\prod\limits_{\ell=2}^k \dfrac{(\alpha_\ell)!}{2^{\alpha_\ell/2}(\alpha_\ell/2)!} (2\ell)^{\alpha_\ell/2}, & \text{if all } \alpha_\ell \text{ are even } \\
0, & \text{else }
\end{cases}
\label{eq:moment2}
\end{align}
where the right hand side is the joint moment of the multivariate gaussian distribution $\mathcal{N}(\mathbf{0},\Sigma)$, where $\Sigma$ is a diagonal matrix with diagonal vector $(4, 8, \cdots, 2k)$. We write $\bar C^{\alpha}=\bar C_2^{\alpha_2} \bar C_3^{\alpha_3} \cdots \bar C_k^{\alpha_k}$ for short. Define $\mathcal{K}(\alpha)=\sum_{t=2}^k t \alpha_t$. In order to distinguish different indices in the summation in $\bar C^{\alpha}$, we write $I_s^{t}=(i_1^t,i_2^t,\cdots,i_s^t)$ where $i_\ell^t\in [n]$ are disjoint and $J_s^{t}=(j_1^t,j_2^t,\cdots,j_s^t)$ where $j_\ell^t\in [m]$ are disjoint. This implies that
\begin{align*}
    \E^* \bar C^{\alpha}&=(\frac{1}{\sqrt{mn}})^{\mathcal{K}(\alpha)} \E^* \prod_{s=2}^k \prod_{t=1}^{\alpha_s} \sum_{I_s^t,J_s^t} \prod_{\ell=1}^s (G_{j_\ell^t,i_\ell^t} G_{j_\ell^t,i_{\ell+1}^t}-\frac{2\beta_n}{\sqrt{mn}})\\
    &=(\frac{1}{\sqrt{mn}})^{\mathcal{K}(\alpha)}\sum_{I,J} \E^*\prod_{s=2}^k \prod_{t=1}^{\alpha_s}  \prod_{\ell=1}^s (G_{j_\ell^t,i_\ell^t} G_{j_\ell^t,i_{\ell+1}^t}-\frac{2\beta_n}{\sqrt{mn}})\\
    &:=(\frac{1}{\sqrt{mn}})^{\mathcal{K}(\alpha)}\sum_{I,J} \E^* G_{I,J}
\end{align*}
where $I$ represents all $I_s^t$ where $s$ ranges from $2$ to $k$ and $t$ ranges from $1$ to $\alpha_s$. The same definition holds for $J$. For each fixed $I$ and $J$, we define a series of multigraphs. Specifically, we start with $H_0=([n], E(H_0))$. We write $V(H_0)$ to be the set of all non-isolated vertices. We go over all $I_s^t$ and for any $i_\ell^t\in I_s^t$, we draw an edge $(i_\ell^t,i_{\ell+1}^t)$ in $H_0$. Now we construct other multigraphs. Let $R$ be the set of all $j_\ell^t$ that appears in $J$. Fix any $j\in R$, we construct a multigraph $H_j=([n], E(H_j))$ in the following way. We go over all $I_s^t$ and for any $i_\ell^t\in I_s^t$, if $j_\ell^t=j$, then we draw an edge $(i_\ell^t,i_{\ell+1}^t)$ in $H_j$. For any two pairs of $(I,J)$ and $(I',J')$, we write $(I,J)\sim (I',J')$ if $|R(I,J)|=|R(I',J')|$ and there is a permutation $\sigma$ of $[m]$ such that all the multigraphs $H_j(I,J)$ is isomorphic to $H_{\sigma(j)}(I',J')$. We refer to the equivalent classes as types and write $(I,J)\in T$ to indicate that they belong to a specific type $T$. If $(I,J)$'s belong to the same type, then an easy observation tells us that $\E^* G_{I,J}$ are the same by symmetry. Therefore, we can write 
\begin{align*}
    \E^* \bar C^\alpha = (\frac{1}{\sqrt{mn}})^{\mathcal{K}(\alpha)}\sum_{T}\sum_{I,J\in T} \E^* G_{I,J} 
    =\sum_{T}\mathcal{O}\left((\frac{1}{n})^{\mathcal{K}(\alpha)} n^{|V(H_0)|}n^{|R|} \right) \E^*G_{T}.
\end{align*}
If we denote $R_1$ as the number of even multigraphs $H_j$, $R_2$ as the number of multigraphs $H_j$ with exactly two odd vertices and $|E(H_j)|\geq 2$ and $R_3=|R|-R_1-R_2$, then by Lemma \ref{l:rows}, we have that $\E^* G_T$ is bounded by $\mathcal{O}(n^{-R_2-2R_3})$. This implies that
\begin{align*}
    \E^* \bar C^\alpha = \sum_{T} \mathcal{O}(n^{-\mathcal{K}(\alpha)+|V(H_0)|+R_1+R_2+R_3-R_2-2R_3})=\sum_{T} \mathcal{O}(n^{-\mathcal{K}(\alpha)+|V(H_0)|+R_1-R_3}).
\end{align*}
Now for any $v\in V(H_0)$, we define $R_1(v)$ to be the number of $j$ such that $H_j$ is even and $v\in V(H_j)$. Then we claim that $\deg_{H_0}(v)\geq R_1(v)+2$. Note that as $H_0$ is a sum of cycles, any vertex $v\in V(H_0)$ satisfies $\deg_{H_0}(v)\geq 2$. Moreover, as $H_j$ are even, $\deg_{H_0}(v)\geq 2R_1(v)$. Now if $R_1(v)=1$, then two edges $(v,u_1)$ and $(v,u_2)$ in $E(H_j)$ must come from two different cycles by definition. In each cycle, there is another edge that are incident to $v$. This implies that $\deg_{H_0}(v)\geq 4$. All together, this claim holds. This implies that
\begin{align*}
    \frac{1}{2}\sum_{v\in V(H_0)} \deg(v)\geq \frac{1}{2}\sum_{v\in V(H_0)} R_1(v)+V(H_0)\geq R_1+V(H_0).
\end{align*}
Therefore,
\begin{align*}
    \E^* \bar C^\alpha =\sum_T \mathcal{O}(n^{-R_3}).
\end{align*}
Note that in order for a type $T$ to have $\Theta(1)$ in the summation, we will need all the inequalities above to be equality. Specifically, we need $R_3=0$, $\deg_{H_0}(v)= R_1(v)+2$ for any vertex $v\in V(H_0)$ and $|V(H_j)|=2$ for any even $H_j$. Note that $\deg_{H_0}(v)= R_1(v)+2$ holds only when $\deg_{H_0}(v)=2$ and $R_1(v)=0$ or $\deg_{H_0}(v)=4$ and $R_1(v)=2$. Now we show that $R_2=0$. If $R_2\neq 0$, find any $H_j$ with exactly two odd vertices and $|E(H_j)|\geq 2$. Find a vertex $v\in V(H_j)$ with $\deg_{H_j}(v)$ even. Then $\deg_{H_0}(v)>2$ because any two edges $(v,u_1)$ and $(v,u_2)$ in $E(H_j)$ must come from two different cycles. Yet, $\deg_{H_0}(v)=4$ and $R_1(v)=2$ cannot hold at the same time because $\deg_{H_0}(v)\geq 2R_1(v)+R_2(v)$. Therefore, all $H_j$ are even graphs and as $|V(H_j)|=2$ for any even $H_j$, $H_j$ can only be a double edge. Moreover, $H_0$ must be a disjoint union of double cycles (cycles where each edge is a double edge). This is only possible when all $\alpha_s$ are even and by definition, there should be $\alpha_s/2$ many $s$-double cycles. We call this type $T_0$. Any other type of graphs will give rise to $\mathcal{O}(n^{-1})$ in the summation. Now we compute 
\begin{align*}
    (\frac{1}{\sqrt{mn}})^{\mathcal{K}(\alpha)}\sum_{I,J\in T_0} \E^* G_{I,J} &= (\frac{1}{\sqrt{mn}})^{\mathcal{K}(\alpha)} \sum_{I,J\in T_0} (1+\mathcal{O}(n^{-1})) \\
    &= (\frac{1}{\sqrt{mn}})^{\mathcal{K}(\alpha)} \#\{I,J\in T_0\} (1+\mathcal{O}(n^{-1})).
\end{align*}
Note that for each $s$, there are  $(\alpha_s)!/(2^{\alpha_s/2}(\alpha_s/2)!)$ ways to find a perfect matching of the $\alpha_s$ cycles and for each pair of cycles, there are $2s$ ways to align them into a double cycle. Additionally, there are $n^{\mathcal{K}(\alpha)/2}(1+\mathcal{O}(n^{-1}))$ ways to find $\mathcal{K}(\alpha)/2$ vertices in $H_0$ and $m^{\mathcal{K}(\alpha)/2}(1+\mathcal{O}(n^{-1}))$ ways to find $\mathcal{K}(\alpha)/2$ rows to form $R$. Putting all of these together, we have that
\begin{align*}
    (\frac{1}{\sqrt{mn}})^{\mathcal{K}(\alpha)}\sum_{I,J\in T_0} \E^* G_{I,J}=
   (1+\mathcal{O}(n^{-1}))\prod\limits_{s=2}^k \dfrac{(\alpha_s)!}{2^{\alpha_s/2}(\alpha_s/2)!} (2s)^{\alpha_s/2}.
\end{align*}
Therefore, equation \eqref{eq:moment2} holds. By the method of moment, we have Lemma \ref{l:normal}. The argument for $\bP^{*2}_t$ and $\bP$ are the same, so we omit it here.
\end{proof}

\begin{proof}[Proof of Lemma \ref{l:cycleRterm}]
We start with $\bP^*$. Notice that 
\begin{align*}
    &C_k(G)-\bar C_k(G)-(2\beta_n)^k\\
    &=-(2\beta_n)^k+(\frac{1}{\sn})^k(\frac{1}{\sm})^k\sum_{\substack{i_1,i_2,\cdots, i_k\in [n] \text{ disjoint}\\j_1,j_2,\cdots, j_k\in [m] \text{ disjoint}}}\prod_{\ell=1}^k (G_{j_\ell,i_\ell} G_{j_\ell,i_{\ell+1}})-\prod_{\ell=1}^k (G_{j_\ell,i_\ell} G_{j_\ell,i_{\ell+1}}-\frac{2\beta_n}{\sqrt{mn}})\\
    &=(\frac{1}{\sn})^k(\frac{1}{\sm})^k\sum_{\substack{i_1,i_2,\cdots, i_k\in [n] \text{ disjoint}\\j_1,j_2,\cdots, j_k\in [m] \text{ disjoint}}} \sum_{A\subsetneq[k],A\neq \emptyset}\prod_{\ell\in A} (G_{j_\ell,i_{\ell}}G_{j_\ell,i_{\ell+1}}-\frac{2\beta_n}{\sqrt{mn}})\prod_{{\ell\in [k]\backslash A}} \frac{2\beta_n}{\sqrt{mn}}+\mathcal{O}(\frac{1}{n}).
\end{align*}
Write $I=i_1,\cdots,i_k$ and $J=j_1,\cdots,j_k$. Define
\begin{align*}
    R(I,J):=\sum_{A\subsetneq[k],A\neq \emptyset}\prod_{\ell\in A} (G_{j_\ell,i_{\ell}}G_{j_\ell,i_{\ell+1}}-\frac{2\beta_n}{\sqrt{mn}})\prod_{{\ell\in [k]\backslash A}} \frac{2\beta_n}{\sqrt{mn}}.
\end{align*}
Note that
\begin{align}
    &\E^* (C_k(G)-\bar C_k(G)-(2\beta_n)^k)^2 \label{eq:Rterm}\\
    &=(\frac{1}{n})^k(\frac{1}{m})^k \sum_{I_1,J_1,I_2,J_2} \E^*\mathcal{R}(I_1,J_1) \mathcal{R}(I_2,J_2) + \mathcal{O}(\frac{1}{n}).
\end{align}
Notice that
\begin{align*}
    &\E^*\mathcal{R}(I_1,J_1) \mathcal{R}(I_2,J_2)\\
    &=\sum_{\substack{A_1\subsetneq[k]\\A_1\neq \emptyset}}\sum_{\substack{A_2\subsetneq[k]\\A_1\neq \emptyset}} \mathcal{O}(n^{|A_1|+|A_2|-2k}) \E^*\prod_{\ell\in A_1} (G_{j_\ell^1,i_{\ell}^1}G_{j_\ell^1,i_{\ell+1}^1}-\frac{2\beta_n}{\sqrt{mn}}) \prod_{\ell\in A_2} (G_{j_\ell^2,i_{\ell}^2}G_{j_\ell^2,i_{\ell+1}^2}-\frac{2\beta_n}{\sqrt{mn}}).
\end{align*}
For a fixed $A_1$ and $A_2$, we construct multigraphs $H_0$ and $H_j$ the same way as before for $i_\ell^1,j_\ell^1$ where $\ell\in A_1$ and $i_\ell^2,j_\ell^2$ where $\ell\in A_2$. Using the same notation, by Lemma \ref{l:rows}, we have
\begin{align*}
    \E^*\prod_{\ell\in A_1} (G_{j_\ell^1,i_{\ell}^1}G_{j_\ell^1,i_{\ell+1}^1}-\frac{2\beta_n}{\sqrt{mn}}) \prod_{\ell\in A_2} (G_{j_\ell^2,i_{\ell}^2}G_{j_\ell^2,i_{\ell+1}^2}-\frac{2\beta_n}{\sqrt{mn}})=\mathcal{O}(n^{-R_2-2R_3}).
\end{align*}
And thus its total contribution to the sum in \eqref{eq:Rterm} is $\mathcal{O}(n^{|A_1|+|A_2|-4k+V(H_0)+R_1-R_3})$.
Note that any even graph $H_j$ must have two vertices (thus it is a double edge), because our $A_1$ and $A_2$ are only coming from two cycles. If $H_j$ has more than two vertices, then any vertices in $H_j$ with degree larger than $2$ will need more than two cycles and thus cause a contradiction. So $R_1$ is bounded by the number of double edges in $H_0$. As $H_0$ are sum of two subsets of a cycle, $E(H_0)$ consists of only edges and double edges. Therefore, $R_1+V(H_0)\leq 2k$. As $|A_1|<k$ and $|A_2|<k$, we have that
\begin{align*}
   \E^{*}(C_k(G)-\bar C_k(G)-(2\beta_n)^k)^2 =\mathcal{O}(\frac{1}{n}).
\end{align*}
For $\bP^{*2}_t$ when $|t|\leq \log n$, we note that
\begin{align*}
    &(\frac{1}{\sn})^k(\frac{1}{\sm})^k\sum_{\substack{i_1,i_2,\cdots, i_k\in [n] \text{ disjoint}\\j_1,j_2,\cdots, j_k\in [m] \text{ disjoint}}}\prod_{\ell=1}^k \frac{4\beta_n\mathbbm{1}[i_\ell,j_\ell\in Q]}{\sqrt{mn}}\\
    &=(\frac{1}{\sn})^k(\frac{1}{\sm})^k (mn)^k \left((\frac{n}{2}+\frac{t\sqrt{n}}{2})^k+(\frac{n}{2}-\frac{t\sqrt{n}}{2})^k\right)(\frac{4\beta_n}{\sqrt{mn}})^k (1+\mathcal{O}(\frac{1}{n}))\\
    &=2(2\beta_n)^k(1+\mathcal{O}(\frac{t}{\sqrt{n}})).
\end{align*}
The rest of the arguments are similar. We will need to apply Lemma \ref{l:rows2} in replacement of Lemma \ref{l:rows}. We omit the detailed arguments here.
\end{proof}

\subsection{Proof of Lemma \ref{l:secondmoment}}\label{ss:secondmoment}
We firstly note that an immediate consequence of Lemma \ref{l:YMconverge} is the following.
\begin{lem}\label{l:expectbound}
Let $\kappa>0$ and $0<\alpha<\alpha_c(\kappa)$. Take $m= \lfloor\alpha n\rfloor$. For any integer $M_1\geq 2$ and $M_2>0$,
\begin{align*}
    \E^* \exp(-Y_{M_1}\mathbbm{1}[|Y_{M_1}|\leq M_2])\geq 1-\exp(-\frac{M_2^2}{2L(M_1)})
\end{align*}
as $n$ goes to infinity. For any $|t|\leq \log n$ and any integer $M_1\geq 2$,
\begin{align*}
    \E^{*2}_t \exp(-2Y_{M_1}\mathbbm{1}[|Y_{M_1}|\leq M_2])\leq \exp(-\frac{1}{2}L(M_1))(1+o_n(1))
\end{align*}
as $n$ goes to infinity. 
\end{lem}

\begin{proof}[Proof of Lemma \ref{l:secondmoment}]
To begin with, note that $C_{k}$ is invariant if we multiply any column of $G$ with $-1$. This implies that the law of $C_{k}$ under $\bP^*$ is the same for any solutions planted. The same holds for $Y_{M_1}$. Therefore, by Lemma \ref{l:expectbound},
\begin{align*}
    \E \left (\frac{Z(G)}{\exp(Y_{M_1}\mathbbm{1}[|Y_{M_1}|\leq M_2])} \right ) &= 2^{-mn}\sum\limits_{X\in \{\pm 1\}^n}\sum\limits_{G\in \{\pm 1\}^{mn}} \frac{\mathbbm{1}(X\in S(G)) }{\exp(Y_{M_1}\mathbbm{1}[|Y_{M_1}|\leq M_2])}\\
    &=\E(Z(G))\E^* \exp(-Y_{M_1}(G,m,n)\mathbbm{1}[|Y_{M_1}|\leq M_2])\\
    &\geq (1-\varepsilon) \E(Z(G)),
\end{align*}
when $M_2$ is large. Note that $C_{k}$ is also invariant under any permutations of the columns. This implies that the law of $C_{k}$ under $\bP^{*2}_t$ is the same for any pair of solutions planted with $\langle X_1,X_2 \rangle=t\sn$. Similar to the above, we have
\begin{align*}
    &\E \left (\frac{Z(G)^2}{\exp(2Y_{M_1}\mathbbm{1}[|Y_{M_1}|\leq M_2])} \right ) \\
    &= 2^{-mn}\sum\limits_{X_1\in \{\pm 1\}^n}\sum\limits_{X_2\in \{\pm 1\}^n}\sum\limits_{G\in \{\pm 1\}^{mn}} \frac{\mathbbm{1}(X_1,X_2\in S(G)) }{\exp(2Y_{M_1}\mathbbm{1}[|Y_{M_1}|\leq M_2])}\\
    &=2^{-mn}\sum\limits_{\substack{t:|t|\leq \kappa\sn\\ \frac{n}{2}+\frac{t\sn }{2} \in \Z  }}\sum\limits_{\substack{X_1,X_2\in \{\pm 1\}^n\\\langle X_1,X_2 \rangle=t\sn}}\sum\limits_{G\in \{\pm 1\}^{mn}} \frac{\mathbbm{1}(X_1,X_2\in S(G)) }{\exp(2Y_{M_1}\mathbbm{1}[|Y_{M_1}|\leq M_2])}\\
    &=4^n 2^{-n} \sum\limits_{\substack{t:|t|\leq \kappa\sn\\ \frac{n}{2}+\frac{t\sn }{2} \in \Z  }} \binom{n}{\frac{n}{2}+\frac{t\sn }{2}} \mathcal{P}_t \E^{*2}_t \exp(-2Y_{M_1}\mathbbm{1}[|Y_{M_1}|\leq M_2]),
\end{align*}
where $\mathcal{P}_t$ denotes the probability of $X_1,X_2\in S(G)$ for a given pair of vectors $X_1$ and $X_2$ where $\langle X_1,X_2 \rangle=t\sn$. 
Note that $\mathcal{P}_t$ is explicit,
\begin{align}
    \mathcal{P}_t=\left( \sum\limits_{x_1,x_2} \binom{\frac{n}{2}+\frac{t\sn }{2}}{\frac{n}{4}+\frac{(t+x_1+x_2)\sn }{4}}\binom{\frac{n}{2}-\frac{t\sn }{2}}{\frac{n}{4}-\frac{(t+x_1-x_2)\sn }{4}}2^{-n} \right )^m, \label{eq:Pt}
\end{align}
where the summation is over $x_1$ and $x_2$ that satisfy $|x_1|, |x_2|\leq \kappa$, $\frac{n}{2}+\frac{x_1\sn }{2} \in \Z $, $\frac{n}{2}+\frac{x_2\sn }{2} \in \Z$ and $n+\frac{(x_1+x_2)\sn }{2}$ even. By Stirling's approximation, for $t\leq \log n$,
\begin{align*}
    \mathcal{P}_t&=\left(P_{\kappa,n}^2 (1+\frac{(t^2-1)(1-\mu_{2,\kappa,n})^2}{2n}+\mathcal{O}(\frac{t^4}{n^2}))\right)^m\\
    &=P_{\kappa,n}^{2m} \exp(2\beta_n^2 (t^2-1)) (1+o(1)).
\end{align*}
More generally, by Stirling's approximation again, for any small $\varepsilon>0$, there exist $\delta>0$ such that for $|t|\leq \delta \sqrt{n}$, we have
\begin{align*}
    \mathcal{P}_t\leq P_{\kappa,n}^{2m} \exp((2\beta^2+\varepsilon) t^2),
\end{align*}
when $n$ is large. Note that as $-1/2<\beta<0$, for a small $\varepsilon_2$ and for $|t|\leq \delta \sqrt{n}$,
\begin{align*}
    2^{-n}\binom{n}{\frac{n}{2}+\frac{t\sn }{2}}\frac{\mathcal{P}_t}{P_{\kappa,n}^{2m}}\leq \exp(-\varepsilon_2 t^2).
\end{align*}
By using Lemma \ref{h:1}, for small enough $\delta$, we have for $|t|\geq \delta \sqrt{n}$
\begin{align*}
    2^{-n}\binom{n}{\frac{n}{2}+\frac{t\sn }{2}}\frac{\mathcal{P}_t}{P_{\kappa,n}^{2m}}\leq \exp(-\varepsilon_3 n),
\end{align*}
for a small $\varepsilon_3>0$. We also note that $ \E^{*2}_t \exp(-2Y_{M_1}(G,m,n)\mathbbm{1}[|Y_{M_1}|\leq M_2])\leq \exp(2M_2)$. As $-1/2<\beta<0$, by Lemma \ref{l:expectbound}, we can bound the second moment by 
\begin{align*}
    &\E(Z(G))^2  \sum\limits_{\substack{t:|t|\geq
    \delta \sn\\ \frac{n}{2}+\frac{t\sn }{2} \in \Z  }} \exp(-\varepsilon_3 n) \exp(2M_2)
    +\E(Z(G))^2  \sum\limits_{\substack{t:\log n \leq |t|\leq
    \delta \sn\\ \frac{n}{2}+\frac{t\sn }{2} \in \Z  }} \exp(-\varepsilon_2 t^2) \exp(2M_2)    \\
   +& \E(Z(G))^2 \sum\limits_{\substack{t:|t|\leq \log n\\ \frac{n}{2}+\frac{t\sn }{2} \in \Z  }} \binom{n}{\frac{n}{2}+\frac{t\sn }{2}}  \exp(2\beta_n^2 (t^2-1)) (1+o_n(1)) \exp(-\frac{1}{2}L(M_1))\\
   \leq& \E(Z(G))^2 \exp(-2\beta^2)\frac{1}{\sqrt{1-4\beta^2}}\exp(-\frac{1}{2}L(M_1))(1+o_n(1)).
\end{align*}
Notice that when $-1/2<\beta<0$, $L(M_1)$ converges as $M_1$ goes to infinity. In particular,
\begin{align*}
    |L(M_1)-(-\log(1-4\beta^2)-4\beta^2)|\rightarrow 0,
\end{align*}
as $M_1 \rightarrow \infty$. So by taking $M_1$ large, the second moment is bounded by $(1+\varepsilon)\E(Z(G))^2$.
\end{proof}

\subsection{Proof of Theorem \ref{t:lognormal} and Theorem \ref{t:ConvInP}}
\begin{proof}[Proof of Theorem \ref{t:ConvInP}]
Note that
\begin{align*}
    \left|\frac{Z(G)/\E Z}{\exp(Y)}-1\right|
    &\leq \left|\frac{Z(G)/\E Z}{\exp(Y)}-\frac{Z(G)/\E Z}{\exp(Y\mathbbm{1}(|Y|\leq M_2))})\right|\\
    &+\left|\frac{Z(G)/\E Z}{\exp(Y\mathbbm{1}(|Y|
    \leq M_2))}-\frac{Z(G)/\E Z}{\exp(Y_{M_1}\mathbbm{1}(|Y_{M_1}|\leq M_2))}\right|\\
    &+\left|\frac{Z(G)/\E Z}{\exp(Y_{M_1}\mathbbm{1}(|Y_{M_1}|\leq M_2))} -1\right|\\
    &:=T_1+T_2+T_3.
\end{align*}
Notice that $\E C_k=0$ for any $k$. We start by bounding $\E C_k^2$ when $k\leq \log n$. We have that
\begin{align*}
    \E C_k^2=(\frac{1}{n})^k(\frac{1}{m})^k\sum_{\substack{i_1^1,i_2^1,\cdots, i_k^1\in [n] \text{ disjoint}\\j_1^1,j_2^1,\cdots, j_k^1\in [m] \text{ disjoint}}} \sum_{\substack{i_1^2,i_2^2,\cdots, i_k^2\in [n] \text{ disjoint}\\j_1^2,j_2^2,\cdots, j_k^2\in [m] \text{ disjoint}}} \E \prod_{\ell=1}^k G_{j^1_\ell,i^1_\ell} G_{j^1_\ell,i^1_{\ell+1}} \prod_{\ell=1}^k G_{j^2_\ell,i^2_\ell} G_{j^2_\ell,i^2_{\ell+1}}.
\end{align*}
In the summation, the only terms that are not zero are those when $G_{\cdot,\cdot}$ appears exactly twice in the product. Therefore,
\begin{align*}
    \E C_k^2\leq (\frac{1}{n})^k(\frac{1}{m})^k (mn)^{k} 2k=2k.
\end{align*}
This implies that
\begin{align*}
    \E Y^2 
    = \E (\sum_{k=2}^{\lfloor \log n \rfloor} \frac{2(2\beta)^kC_{k}-(2\beta)^{2k}}{4k})^2
    &\leq \frac{1}{(1-4\beta^2)^2}+\sum_{k=2}^\infty |2\beta|^k \sum_{k=2}^\infty |2\beta|^k \frac{\E C_k^2}{(2k)^2}\\
    &\leq \frac{1}{(1-4\beta^2)^2}+\frac{1}{(1-|2\beta|)^2}.
\end{align*}
Moreover,
\begin{align}
    \E |Y-Y_{M_1}|^2
    = \E (\sum_{k=M_1}^{\lfloor \log n \rfloor} \frac{2(2\beta)^kC_{k}-(2\beta)^{2k}}{4k})^2
    &\leq \frac{(2\beta)^{2M_1}}{(1-4\beta^2)^2}+\sum_{k=M_1}^\infty |2\beta|^k \sum_{k=M_1}^\infty |2\beta|^k \frac{\E C_k^2}{(2k)^2}\\
    &\leq \frac{(2\beta)^{2M_1}}{(1-4\beta^2)^2}+\frac{(2\beta)^{2M_1}}{(1-|2\beta|)^2}.\label{eq:diffY}
\end{align}
For any $\varepsilon,\varepsilon'>0$, by Lemma \ref{l:secondmoment}, we know that we can choose $M_1$ and $M_2$ large enough such that
\begin{align*}
    \bP (T_3\geq \varepsilon)<\varepsilon',
\end{align*}
for large $n$. Further, we have
\begin{align*}
    \bP(T_1\geq \varepsilon)\leq \bP(|Y|\geq M_2)\leq (\frac{1}{(1-4\beta^2)^2}+\frac{1}{(1-|2\beta|)^2})\frac{1}{M_2^2}.
\end{align*}
In addition, by Lemma \ref{l:YMconverge}, we have for a large $M_2>0$,
\begin{align*}
    \bP(|Y_{M_1}|\geq M_2)\leq \exp(-\frac{M_2^2}{2L(M_1)}),
\end{align*}
when $n$ is large. Using equation \eqref{eq:diffY}, we further have that for a big constant $C$,
\begin{align*}
    \bP(|Y|\geq M_2)\leq \exp(-\frac{M_2^2}{4L(M_1)})+C(2\beta)^{2M_1}
\end{align*}
when $n$ is large. This implies that when $n$ is large, for a big constant $C$, 
\begin{align*}
    &\bP (\exp(|Y\mathbbm{1}(|Y|
    \leq M_2)-Y_{M_1}\mathbbm{1}(|Y_{M_1}|
    \leq M_2)|)\geq 1+\varepsilon)\\
    &\leq 4M_2(\exp(-\frac{M_2^2}{4L(M_1)})+C(2\beta)^{2M_1})+\frac{C(2\beta)^{2M_1}}{\log( 1+\varepsilon)^2},
\end{align*}
which goes to zero when $M_1$ and $M_2$ are large. Therefore, for any $\varepsilon,\varepsilon'>0$, we can choose $M_1$ and $M_2$ large enough such that 
\begin{align*}
    \bP(T_2\geq \frac{\varepsilon Z(G)/\E Z}{\exp(Y_{M_1}\mathbbm{1}(|Y_{M_1}|\leq M_2))})\leq \varepsilon'. 
\end{align*}
Together with our estimation on $T_3$, we have the theorem.
\end{proof}
\begin{proof}[Proof of Theorem \ref{t:lognormal}]
By equation \eqref{eq:diffY}, we know that $\E|Y-Y_{M_1}|^2$ can be made arbitrarily small by taking $M_1$ large. By Lemma \ref{l:YMconverge}, under $\bP$, for any integer $M_1\geq 2$,
\begin{align*}
    Y_{M_1}\xrightarrow{d} \mathcal{N}(-\frac{1}{4}L(M_1),\frac{1}{2}L(M_1)),
\end{align*}
as $n$ goes to infinity. As $L(M_1)\to -\log(1-4\beta^2)-4\beta^2$ when $M_1$ converges to infinity, we have that
\begin{align*}
    Y\xrightarrow{d} \mathcal{N}(\frac{1}{4}\log(1-4\beta^2)+\beta^2,-\frac{1}{2}\log(1-4\beta^2)-2\beta^2),
\end{align*}
as $n$ goes to infinity. This implies that
\begin{align*}
    \exp(Y)\xrightarrow{d} \mathrm{Lognormal}(\frac{1}{4} \log(1-4\beta^2)+\beta^2,-\frac{1}{2}\log(1-4\beta^2)-2\beta^2),
\end{align*}
as $n$ goes to infinity. Together with Theorem \ref{t:ConvInP}, we have the desired result.
\end{proof}

\subsection{Proof of Theorem \ref{t:contiguity}}
Consider a series of events $A_n$. If $\bP(A_n) \to 0$, then 
\begin{align*}
    \bP^*(A_n)=\E [ \frac{Z(G)\mathbbm{1}(A_n)}{\E Z}]\leq \E [ \frac{Z(G)}{\E Z} \mathbbm{1}(\frac{Z(G)}{\E Z}\geq \exp (M_2))]+\exp(M_2)P(A_n).
\end{align*}
When $M_2$ goes to infinity, the first term goes to zero. Therefore, for any $\varepsilon>0$, we choose an $M_2$ such that the first term is bounded by $\varepsilon/2$. Then there exist an integer $N>0$ such that $\bP^*(A_n)<\varepsilon$ whenever $n>N$. Now, if $\bP^*(A_n) \to 0$, then by Theorem \ref{t:lognormal},
\begin{align*}
    \bP(A_n)&\leq \bP(\frac{Z(G)}{\E Z}\leq \exp(-M_2))+\bP(A_n\mathbbm{1}(\frac{Z(G)}{\E Z}\geq \exp(-M_2)))\\
    &\leq \bP(\frac{Z(G)}{\E Z}\leq \exp(-M_2))+\E^*[\frac{\E Z}{ Z(G)}\mathbbm{1}(\frac{Z(G)}{\E Z}\geq \exp(-M_2),A_n)]\\
    &\leq \exp(-cM_2^2)+\exp(M_2)\bP^*(A_n),
\end{align*}
for some constant $c$ when $M_2$ and $n$ are large. This implies that for any $\varepsilon>0$, we can choose $M_2$ large such that the first term is bounded by $\varepsilon/2$. Then there exist an integer $N>0$ such that $\bP(A_n)<\varepsilon$ whenever $n>N$. Therefore, the theorem follows.

\subsection{Proof of Theorem \ref{t:threshold}}
For $\kappa>0$ and $\alpha>\alpha_c$, take $m=\lfloor \alpha n \rfloor$. Let $0<\delta<\alpha-\alpha_c$. Then when $n$ is large,
\begin{align*}
    \bP(Z(G)\geq 1)\leq \E(Z(G))=2^n P_\kappa^m=\exp(n(\log 2 + \frac{m}{n}\log (P_\kappa)))\leq \exp(-\delta n),
\end{align*}
which converges to zero as $n$ goes to infinity. For $\kappa>0$ and $\alpha<\alpha_c$, take $m=\lfloor \alpha n \rfloor$. Let $0<\delta<\alpha_c-\alpha$. Then,
\begin{align*}
    \E(Z(G))=2^n P_\kappa^m=\exp(n(\log 2 + \frac{m}{n}\log (P_\kappa)))\geq \exp(\delta n).
\end{align*}
This implies that
\begin{align*}
    \bP(Z(G)\geq 1)=\bP\left(\frac{Z(G)}{\E Z(G)}\geq \frac{1}{\E Z(G)}\right) \geq \bP\left(\frac{Z(G)}{\E Z(G)}\geq \exp(-\delta n)\right),
\end{align*}
which converges to $1$ as $n$ goes to infinity by Theorem \ref{t:lognormal}.

\subsection{Proof of Theorem \ref{t:freezing}}
Let $X$ be chosen uniformly at random from $\{\pm 1\}^n$. Consider the planted model $\bP^*$ when $X$ is planted. We firstly show that 
\begin{align*}
    \lim_{n\rightarrow \infty}\bP^*(\{X_2\in S(G):d(X,X_2)\leq \delta n)\}=\{X\})=1.
\end{align*}
Recall that we defined $\mathcal{P}_t$ to be the probability of $X_1,X_2\in S(G)$ for a given pair of vectors $X_1$ and $X_2$ where $\langle X_1,X_2 \rangle=t\sn$. Then we have
\begin{align*}
    \E^*(|\{X_2\in S(G):d(X_2,X)\leq \delta n\}|-1)=\sum_{\ell=1}^{\lfloor\delta n\rfloor} \binom{n}{\ell} \frac{\mathcal{P}_{(n-2\ell)/\sn}}{P_{\kappa,n}^m}.
\end{align*}
We start by showing that there exists a small $\delta_2>0$ and $c>0$ such that 
\begin{align*}
    \mathcal{P}_{(n-2\ell)/\sn}\leq \exp(-c\sqrt{\ell n})P_{\kappa,n}^m, 
\end{align*}
whenever $\ell\leq \delta_2 n$. Without loss of generality, assume that $X_1$ is the all one vector and $X_2$ is the vector with $\ell$ ones in the front and $-1$ in the rest of the entries. Let $X_0$ be chosen uniformly at random from $\{\pm 1\}^n$. Define $E_1$ to be the event that
$|\langle X_1,X_0 \rangle|\leq \kappa \sn$ and $E_2$ to be the event that
$|\langle X_2,X_0 \rangle|\leq \kappa \sn$.
For large $n$ and small $\delta_2$, if $\ell\leq \delta_2 n$, then there exists a small constant $c$ such that
\begin{align*}
    \bP(E_1)-\bP(E_1,E_2)\geq \bP\left(0\leq \kappa\sn-B_2\leq \sqrt{\ell},B_1\geq \sqrt{\ell}\right)\geq c\sqrt{\ell}/\sqrt{n},
\end{align*}
where $B_1$ and $B_2$ are independent, $B_1$ is the sum of $\ell$ independent Rademacher random variables and $B_2$ is the sum of $n-\ell$ independent Rademacher random variables. This implies that $\bP(E_1,E_2)\leq \exp(-c\sqrt{\ell}/\sqrt{n})P_{\kappa,n}$. Further we have for small constants $c_1$ and $c_2$,
\begin{align*}
    \mathcal{P}_{(n-2\ell)/\sn}\leq \exp(-c_1\sqrt{\ell}m/\sqrt{n})P_{\kappa,n}^m\leq \exp(-c_2\sqrt{\ell n})P_{\kappa,n}^m.
\end{align*}
Therefore,
\begin{align*}
    \sum_{\ell=1}^{\lfloor\delta n\rfloor} \binom{n}{\ell} \frac{\mathcal{P}_{(n-2\ell)/\sn}}{P_{\kappa,n}^m}\leq \sum_{\ell=1}^{\lfloor\delta n\rfloor} \binom{n}{\ell}\exp(-c_2\sqrt{\ell n})=o_n(1),
\end{align*}
when $\delta$ is small. By Theorem \ref{t:contiguity}, the statement holds.

\subsection{Properties of the special function}\label{ss:hypo}
\subsubsection{Proof of Lemma \ref{h:1}}
To prove this lemma, we firstly note that $F_{\kappa,\alpha}$ is symmetric around $1/2$, so it only suffices to focus on $x\in [0,1/2]$. Further, notice that $q_\kappa(0)=P_\kappa$ and $q_\kappa(1/2)=P_\kappa^2$. This implies that for any $x\in [0,1/2)$ and $0<\alpha<\alpha_c(\kappa)$,
\begin{align*}
    \left(F_{\kappa, \alpha} (\frac{1}{2}) - F_{\kappa, \alpha} (x) \right) - \left(F_{\kappa, \alpha_c} (\frac{1}{2}) - F_{\kappa, \alpha_c} (x) \right)= (\alpha-\alpha_c) \log \left( \frac{P_\kappa^2}{q_\kappa(x)}\right)>0.
\end{align*}
Therefore, in order to prove Lemma \ref{h:1}, it suffices to show that $F_{\kappa, \alpha_c}(x)\leq F_{\kappa, \alpha_c}(1/2)$ for any $x\in [0,1/2]$. For simplicity of notations, we sometimes write $F$ and $\alpha_c$ to denote $F_{\kappa, \alpha_c}$ and $\alpha_c(\kappa)$ respectively.

We prove this statement in a few steps. We start by showing that the inequality holds for any $\kappa>4$ (Lemma \ref{l:hypostep1}), followed by showing that the inequality holds for any $\kappa<0.1$ (Lemma \ref{l:hypostep2}). We then show that for $\kappa\in [0.1,4]$, when $x\in [0.3,0.5)$ (Lemma \ref{l:hypostep2.2}) and when $x\in [0,0.005]$ (Lemma \ref{l:hypostep2.3}), the inequality holds. Finally, we check that the inequality holds when $\kappa\in [0.1,4]$ and $x\in [0.005,0.3]$ (Lemma \ref{l:hypostep3}). The proofs of Lemma \ref{l:hypostep2.2} and \ref{l:hypostep3} require searches over certain grid points. We show in Section \ref{sss:specialfn} that all the computations involved are among rational numbers and thus can be evaluated exactly. The proofs of other lemmas are analytical.

\begin{lem}\label{l:hypostep1}
For any $\kappa>4$, $F_{\kappa, \alpha_c}(x)\leq F_{\kappa, \alpha_c}(1/2)$ for any $x\in [0,1/2]$. 
\end{lem}
\begin{proof}
We firstly note that 
\begin{align*}
    F(0)=\frac{-\log(2)}{ \log(P_\kappa)} \log (P_\kappa) = -\log(2) = \log(2) + \frac{-\log(2)}{ \log(P_\kappa)} \log (P_\kappa^2)= F(1/2). 
\end{align*}
$F$ is continuous on $[0,1]$ and smooth on $(0,1)$ with
\begin{align*}
    F'(x)=-\log(x)+\log(1-x)+\alpha_c \frac{-\exp(-\frac{\kappa^2}{2(1-x)})+\exp(-\frac{\kappa^2}{2x})}{\pi \sqrt{x(1-x)}q_\kappa(x)}.
\end{align*}
We will firstly show that when $x\in [0.2,0.5)$, $F'(x)<0$. Define a function
\begin{align*}
    \mathcal{L}(x)=\frac{\alpha_c}{\pi} \frac{\exp(-\frac{\kappa^2}{2(1-x)})-\exp(-\frac{\kappa^2}{2x})}{ \sqrt{x(1-x)}(\log(1-x)-\log(x))}.
\end{align*}
We analyze $\mathcal{L}(x)$ term by term. By Taylor expansion, it is easy to see that
\begin{equation}\label{eq:log1}
    \frac{\log(1-x)-\log(x)}{-(x-1/2)}\geq 4,
\end{equation}
for any $x\in [0.2,0.5]$. Moreover,
\begin{align*}
    \frac{d^2}{dx^2}\left(\exp\left(-\frac{\kappa^2}{2x}\right)\right)=\frac{\kappa^2 \exp(-\frac{\kappa^2}{2x})(\kappa^2-4x)}{4x^4}>0,
\end{align*}
when $\kappa>4$ and $x\leq 1/2$. This implies that
\begin{equation}\label{eq:exp1}
    \frac{\exp(-\frac{\kappa^2}{2(1-x)})-\exp(-\frac{\kappa^2}{2x})}{-(x-1/2)}\leq \frac{\kappa^2 \exp(-\frac{\kappa^2}{2(1-x)})}{(1-x)^2}.
\end{equation}
We also note that
\begin{align*}
    1-P_\kappa=1-\erf(\kappa/\sqrt{2})\geq \exp(-\kappa^2/2)\sqrt{\frac{2}{\pi}}(\frac{1}{\kappa}-\frac{1}{\kappa^3}).
\end{align*}
This implies that 
\begin{equation}\label{eq:alphac}
    \alpha_c=\frac{-\log(2)}{\log(P_\kappa)}\leq \frac{\log(2)}{1-P_\kappa}\leq \log(2)\sqrt{\frac{\pi}{2}}\exp(\kappa^2/2)\kappa \frac{16}{15},
\end{equation}
for $\kappa>4$. Combining equations \eqref{eq:log1}, \eqref{eq:exp1} and \eqref{eq:alphac}, we have for $\kappa>4$ and $x\in [0.35,0.5)$,
\begin{align*}
    \mathcal{L}(x)&\leq \frac{1}{4} \log(2)\sqrt{\frac{\pi}{2}}\exp(\kappa^2/2)\kappa \frac{16}{15} \frac{\kappa^2 \exp(-\frac{\kappa^2}{2(1-x)})}{(1-x)^2} \frac{1}{\sqrt{0.2 \cdot 0.8}}\\
    &\leq \exp\left(\frac{\kappa^2}{2}-\frac{\kappa^2}{2(1-x)}\right)\frac{\kappa^3}{(1-x)^2}0.19\\
    &\leq \exp\left(\frac{\kappa^2}{2}-\frac{\kappa^2}{1.3}\right)\frac{\kappa^3}{{0.5}^2}0.19
    \leq \exp\left(\frac{4^2}{2}-\frac{4^2}{1.3}\right)\frac{4^3}{{0.5}^2}0.19 < 0.9,
\end{align*}
where the second last inequality follows from the fact that $\exp(x^2/2-x^2/1.3)x^3$ is decreasing when $x>4$. When $x\in [0.2,0.35]$, we use a simple inequality
\begin{align*}
    \frac{\exp(-\frac{\kappa^2}{2(1-x)})-\exp(-\frac{\kappa^2}{2x})}{-(x-1/2)}\leq \frac{\exp(-\frac{\kappa^2}{2(1-x)})}{-(x-1/2)}.
\end{align*}
Combining this with equations \eqref{eq:log1} and \eqref{eq:alphac}, we have
\begin{align*}
    \mathcal{L}(x)&\leq \frac{1}{4} \log(2)\sqrt{\frac{\pi}{2}}\exp(\kappa^2/2)\kappa \frac{16}{15} \frac{\kappa^2 \exp(-\frac{\kappa^2}{2(1-x)})}{-(x-1/2)} \frac{1}{\sqrt{0.2 \cdot 0.8}}\\
    &\leq \exp\left(\frac{\kappa^2}{2}-\frac{\kappa^2}{2(1-x)}\right)\frac{\kappa}{-(x-1/2)}0.19\\
    &\leq \exp\left(\frac{\kappa^2}{2}-\frac{\kappa^2}{1.6}\right)\frac{\kappa}{0.15}0.19
    \leq \exp\left(\frac{4^2}{2}-\frac{4^2}{1.6}\right)\frac{4}{0.15}0.19<0.9,
\end{align*}
for any $x\in [0.2,0.35]$, when $\kappa>4$. The second last inequality follows from the fact that $\exp(x^2/2-x^2/1.6)x$ is decreasing when $x>4$. Moreover, we have that
\begin{align*}
    q_\kappa(x)\geq q_4(0.2)>0.9,
\end{align*}
for any $\kappa>4$ and $x\in[0.2,0.5)$. Therefore, we have for any $x\in [0.2,0.5)$ and $\kappa>4$,
\begin{align*}
    \mathcal{L}(x)=\frac{\alpha_c}{\pi} \frac{\exp(-\frac{\kappa^2}{2(1-x)})-\exp(-\frac{\kappa^2}{2x})}{ \sqrt{x(1-x)}(\log(1-x)-\log(x))}<q_\kappa(x),
\end{align*}
which is equivalent to 
\begin{align*}
    F'(x)>0.
\end{align*}
Next, we will show that $\mathcal{L}'(x)<q_\kappa'(x)$ for any $x\in (0,0.2]$. This implies that there exists some small constant $c_\kappa\in (0,0.2)$ such that $F'(x)<0$ when $x\in (0,c_\kappa)$ and $F'(x)>0$ when $x\in (c_\kappa,0.2)$. (This is because both $\mathcal{L}(x)$ and $q_\kappa(x)$ are positive and $\mathcal{L}(x)$ goes to infinity as $x$ approaches $0$.) In particular, together with above, this shows that $F(x)<F(0)=F(0.5)$ for $x\in (0,0.5)$. Now, we prove this statement. Note that
\begin{align*}
    -\mathcal{L}'(x)&=\frac{\alpha_c}{\pi} \left( \frac{\frac{\kappa^2 e^{-\frac{\kappa^2}{2x}}}{2x^2}   +\frac{\kappa^2 e^{-\frac{\kappa^2}{2(1-x)}}}{2(1-x)^2}  }{\sqrt{x(1-x)}(\log(1-x)-\log(x))}  -  \frac{ (1-2x)( e^{-\frac{\kappa^2}{2x}} - e^{-\frac{\kappa^2}{2(1-x)}})}{2(x(1-x))^{3/2}(\log(1-x)-\log(x))} \right.\\
    & \left.\qquad\qquad -\frac{ (-\frac{1}{x}-\frac{1}{1-x})( e^{-\frac{\kappa^2}{2x}} - e^{-\frac{\kappa^2}{2(1-x)}})}{\sqrt{x(1-x)}(\log(1-x)-\log(x))^2} \right).
\end{align*}
When $\kappa>4$ and $x<0.2$, $\exp(-\kappa^2/(2x))<0.01 \exp(-\kappa^2/(2(1-x)))$. This implies that
\begin{align*}
    -\mathcal{L}'(x)&\geq 0.99\frac{\alpha_c}{\pi} e^{-\frac{\kappa^2}{2(1-x)}} \left( \frac{\frac{\kappa^2 }{2(1-x)^2}  }{\sqrt{x(1-x)}(\log(1-x)-\log(x))}  +  \frac{ (1-2x)}{2(x(1-x))^{3/2}(\log(1-x)-\log(x))} \right.\\
    & \left.\qquad\qquad +\frac{ (-\frac{1}{x}-\frac{1}{1-x})}{\sqrt{x(1-x)}(\log(1-x)-\log(x))^2} \right)\\
    &= \frac{0.99\frac{\alpha_c}{\pi}e^{-\frac{\kappa^2}{2(1-x)}}}{\sqrt{x(1-x)}(\log(1-x)-\log(x))} \left( \frac{\kappa^2 }{2(1-x)^2}  +\frac{1-2x}{2x(1-x)}-\frac{1}{x(1-x)(\log(1-x)-\log(x))}\right).
\end{align*}
At the same time,
\begin{align*}
    -q_\kappa'(x)=\frac{\exp(-\frac{\kappa^2}{2(1-x)})-\exp(-\frac{\kappa^2}{2x})}{\pi \sqrt{x(1-x)}}\leq \frac{\exp(-\frac{\kappa^2}{2(1-x)})}{\pi \sqrt{x(1-x)}}.
\end{align*}
To show $\mathcal{L}'(x)<q_\kappa'(x)$, it suffices to show for $x \in (0,0.2]$ and $\kappa>4$,
\begin{align*}
    &\frac{\kappa^2 }{2(1-x)^2(\log(1-x)-\log(x))}  +\frac{1-2x}{2x(1-x)(\log(1-x)-\log(x))}-\frac{1}{x(1-x)(\log(1-x)-\log(x))^2}\\
    &\geq \frac{1}{0.99\alpha_c}.
\end{align*}
Further, as the left hand side is increasing in $\kappa$ while the right hand side is decreasing in $\kappa$, we only need to show for $x \in (0,0.2]$
\begin{align*}
    &\frac{8 }{(1-x)^2(\log(1-x)-\log(x))}  +\frac{1-2x}{2x(1-x)(\log(1-x)-\log(x))}-\frac{1}{x(1-x)(\log(1-x)-\log(x))^2}\\
    &\geq \frac{1}{10000}.
\end{align*}
Name the three fractions $A_1$, $A_2$ and $A_3$ respectively. When $x<0.05$, $A_2>1.2 A_3$ because $(1-2x)(\log(1-x)-\log(x))/2>1.2$ when $x<0.2$. Therefore,
\begin{align*}
    A_1+A_2+A_3\geq 0.1 \frac{1-2x}{2x(1-x)(\log(1-x)-\log(x))} \geq 0.1 \frac{0.6}{2\cdot 0.2 \cdot 2} >0.01.
\end{align*}
When $x\geq 0.05$,
\begin{align*}
    A_1\geq \frac{8}{0.95^2(\log(0.95)-\log(0.05))} >3.
\end{align*}
And
\begin{align*}
    A_2+A_3=\frac{\frac{1-2x}{2}(\log(1-x)-\log(x))-1}{x(1-x)(\log(1-x)-\log(x))^2} \geq \frac{\frac{1-0.4}{2}(\log(0.8)-\log(0.2))-1}{0.2(1-0.2)(\log(1-0.2)-\log(0.2))^2} >-2.
\end{align*}
To see the second last inequality, if the numerator is non-negative, then the inequality holds trivially. Else, the numerator is decreasing in $x$ and the denominator is concave in $[0.05,0.2]$ where its minimum is achieved at $x=0.2$ (concavity is easily seen by taking a third derivative of the denominator). 
Therefore, $A_1+A_2+A_3>1$ for $x\in [0.05,0.2]$. Combining all the arguments above, we have the lemma. 
\end{proof}

\begin{lem}\label{l:hypostep2}
For any $0<\kappa<0.1$, $F_{\kappa, \alpha_c}(x)\leq F_{\kappa, \alpha_c}(1/2)$ for any $x\in [0,1/2]$. 
\end{lem}
\begin{proof}
We firstly note that
\begin{align*}
    q_\kappa(x)&=\frac{1}{2\pi} \int_{-\kappa}^{\kappa} \int_{\frac{-\kappa+(1-2x)y}{2\sqrt{x(1-x)}}} ^{\frac{\kappa+(1-2x)y}{2\sqrt{x(1-x)}}} \exp\left(-\frac{y^2+z^2}{2}\right) dz dy\\
    &\leq \frac{1}{2\pi} 2\kappa \frac{2\kappa}{2 \sqrt{x(1-x)}} = \frac{\kappa^2}{\pi \sqrt{x(1-x)}}.
\end{align*}
We also note that
\begin{align}
    q_\kappa(x)&=\frac{1}{2\pi} \int_{-\kappa}^{\kappa} \int_{\frac{-\kappa+(1-2x)y}{2\sqrt{x(1-x)}}} ^{\frac{\kappa+(1-2x)y}{2\sqrt{x(1-x)}}} \exp\left(-\frac{y^2+z^2}{2}\right) dz dy\\
    &\geq \frac{1}{2\pi} 2\kappa \frac{2\kappa}{2 \sqrt{x(1-x)}} \exp(-\kappa^2/2x)=\frac{\kappa^2\exp(-\kappa^2/2x)}{\pi \sqrt{x(1-x)}}. \label{eq:boundonq}
\end{align}
Besides, for $\kappa<0.1$, 
\begin{align*}
    P_\kappa = \frac{1}{\sqrt{2\pi}} \int_{-\kappa}^\kappa \exp(-x^2/2) dx \leq \frac{2\kappa}{\sqrt{2\pi}},
\end{align*}
and 
\begin{align*}
    P_\kappa = \frac{1}{\sqrt{2\pi}} \int_{-\kappa}^\kappa \exp(-x^2/2) dx \geq 0.99\frac{2\kappa}{\sqrt{2\pi}}.
\end{align*}
Our proof is divided into several steps. We will show that when $x\in [0.27,0.5)$, $\mathcal{L}(x)<q_\kappa(x)$; when $x\in (0, \kappa^2/2]$, $\mathcal{L}(x)>q_\kappa(x)$; when $x\in (\kappa^2/2,\kappa/12]$, $\mathcal{L}(x)>q_\kappa(x)$; when $x \in [0.14,0.27)$, $F_\kappa(x)<F_\kappa(1/2)$; when $x \in [0.04,0.14)$, $F_\kappa(x)<F_\kappa(1/2)$; when $x \in [\kappa/12,0.04)$, $F_\kappa(x)<F_\kappa(1/2)$. These results combined will give us the lemma. 

We will firstly show that for any $x\in [0.27,0.5)$, $\mathcal{L}(x)<q_\kappa(x)$. By Taylor expansion, it is easy to see that
\begin{align*}
    \frac{\log(1-x)-\log(x)}{-(x-1/2)}\geq 4,
\end{align*}
for any $x\in [0.27,0.5)$. Moreover,
\begin{align*}
    \frac{d^2}{dx^2}\left(\exp\left(-\frac{\kappa^2}{2x}\right)\right)=\frac{\kappa^2 \exp(-\frac{\kappa^2}{2x})(\kappa^2-4x)}{4x^4}<0,
\end{align*}
when $\kappa<0.01$ and $x\geq 0.27$. This implies that
\begin{align*}
    \frac{\exp(-\frac{\kappa^2}{2(1-x)})-\exp(-\frac{\kappa^2}{2x})}{-(x-1/2)}\leq \frac{\kappa^2 \exp(-\frac{\kappa^2}{2x})}{x^2}.
\end{align*}
Therefore, by equation \eqref{eq:boundonq},
\begin{align*}
    \mathcal{L}(x)\leq \frac{\alpha_c}{\pi} \frac{1}{\sqrt{x(1-x)}} \frac{\kappa^2}{x^2}\exp(-\frac{\kappa^2}{2x}) \frac{1}{4} \leq \frac{\alpha_c}{4x^2}q_{\kappa}(x).
\end{align*}
When $\kappa<0.1$ and $x\geq 0.27$, we have $\alpha_c\leq 0.274$ and thus $\alpha_c/4x^2<1$. This implies that $\mathcal{L}(x)<q_\kappa(x)$. Next we show that when $x\in (0, \kappa^2/2]$, $\mathcal{L}(x)>q_\kappa(x)$. Note that
\begin{align*}
    -\log(P_\kappa) \mathcal{L}(x)&=\frac{\log(2)}{\pi} \frac{\exp(-\frac{\kappa^2}{2(1-x)})-\exp(-\frac{\kappa^2}{2x})}{ \sqrt{x(1-x)}(\log(1-x)-\log(x))} \\
    &\geq \frac{\log(2)}{\pi} \frac{1/2}{\sqrt{x}(-\log(x))}\geq \frac{\log(2)}{\pi} \frac{1/2}{\sqrt{0.005}(-\log(0.005))}>0.29,
\end{align*}
as $-\sqrt{x}\log(x)$ is increasing when $x\leq 0.005$. At the same time
\begin{align*}
    -\log(P_\kappa) q_\kappa(x) \leq -\log(P_\kappa) P_\kappa \leq -\log(P_{0.1}) P_{0.1}< 0.21,
\end{align*}
as $-x\log(x)$ is increasing when $x<0.08$. These combined gives $\mathcal{L}(x)>q_\kappa(x)$ when $x\in (0, \kappa^2/2]$. For the next step, we show that $\mathcal{L}(x)>q_\kappa(x)$ when $x\in ( \kappa^2/2, \kappa/12]$. Note that by Taylor expansion, when $x\in ( \kappa^2/2, \kappa/12]$,
\begin{align*}
    \exp(-\frac{\kappa^2}{2(1-x)})-\exp(-\frac{\kappa^2}{2x})\geq -\frac{\kappa^2}{2(1-x)}+\frac{\kappa^2}{2x}-\frac{\kappa^4}{8x^2} > \frac{\kappa^2}{5x}.
\end{align*}
This implies that
\begin{align*}
    \mathcal{L}(x)&\geq \frac{\log(2)}{\pi} \frac{\kappa^2/5x}{\sqrt{x(1-x)}(-\log(x))}\frac{1}{-\log(P_\kappa)}\\
    &\geq q_\kappa(x)\frac{\log(2)}{5x(-\log(x))(-\log(P_\kappa))}\\
    &\geq q_\kappa(x)\frac{\log(2)}{5\kappa/12(-\log(\kappa/12))(-\log(0.99\frac{2\kappa}{\sqrt{2\pi}}))}\\
    &\geq q_\kappa(x)\frac{\log(2)}{5\cdot 0.1/12(-\log(0.1/12))(-\log(0.99\cdot 0.2/\sqrt{2\pi}))}>q_\kappa(x).
\end{align*}
where we have used the fact that $x(-\log(x))$ is increasing on $(0,0.01]$ and $x (-\log(\frac{x}{12})) (-\log(\frac{1.98x}{\sqrt{2\pi}}))$ is increasing on $(0,0.1]$. Next, we show that if $x\in [0.14,0.27]$, then $F_\kappa(x)< F_{\kappa}(1/2)$. By definition, the inequality is equivalent to 
\begin{align*}
    -x\log(x)-(1-x)\log(1-x)+\alpha_c \log(q_\kappa(x))< \log(2)+\alpha_c(P_\kappa^2)
\end{align*}
which is further equivalent to 
\begin{align*}
    \log\left(\frac{q_\kappa(x)}{P_\kappa^2}\right)< \frac{1}{\alpha_c }( \log(2)+x\log(x)+(1-x)\log(1-x)).
\end{align*}
When $x\in [0.14,0.27]$ and $\kappa<0.1$, the right hand
\begin{align*}
    &\frac{1}{\alpha_c }( \log(2)+x\log(x)+(1-x)\log(1-x)) \\
    &\geq \frac{1}{\alpha_c(0.1) }( \log(2)+0.27\log(0.27)+(1-0.27)\log(1-0.27))>0.4.
\end{align*}
And the left hand side
\begin{align*}
    \log\left(\frac{q_\kappa(x)}{P_\kappa^2}\right)<\log\left(\frac{1}{\sqrt{x(1-x)}0.99^2 \cdot 2}\right) \leq \log\left(\frac{1}{\sqrt{0.14(1-0.14)}0.99^2 \cdot 2}\right)<0.4.
\end{align*}
Therefore, when $x\in [0.14,0.27]$, we have $F_\kappa(x)< F_{\kappa}(1/2)$. In the next step, we repeat the above argument and show that when $x\in [0.04,0.14]$, $F_\kappa(x)< F_{\kappa}(1/2)$. We consider the same argument and have that
\begin{align*}
    &\frac{1}{\alpha_c(\kappa) }( \log(2)+x\log(x)+(1-x)\log(1-x)) \\
    &\geq \frac{1}{\alpha_c(0.1) }( \log(2)+0.14\log(0.14)+(1-0.14)\log(1-0.14))>1.05.
\end{align*}
At the same time,
\begin{align*}
    \log\left(\frac{q_\kappa(x)}{P_\kappa^2}\right)<\log\left(\frac{1}{\sqrt{x(1-x)}0.99^2 \cdot 2}\right) \leq \log\left(\frac{1}{\sqrt{0.04(1-0.04)}0.99^2 \cdot 2}\right)<1.05.
\end{align*}
This implies that when $x\in [0.04,0.14]$, $F_\kappa(x)< F_{\kappa}(1/2)$. In the last step, we consider $x\in [\kappa/12,0.04]$. Consider the same equivalence and we have that 
\begin{align*}
    &\frac{1}{\alpha_c(\kappa) }( \log(2)+x\log(x)+(1-x)\log(1-x)) \\
    &\geq \frac{-\log(P_\kappa)}{\log(2) }( \log(2)+0.04\log(0.04)+(1-0.04)\log(1-0.04))\\
    &> -0.75 \log\left(\frac{2\kappa}{\sqrt{2\pi}}\right)=\log\left(\left(\frac{\sqrt{2\pi}}{2}\right)^{0.75} \frac{1}{\kappa^{0.75}}\right)>\log\left(\frac{2}{\sqrt{\kappa}}\right)
\end{align*}
At the same time
\begin{align*}
    \log\left(\frac{q_\kappa(x)}{P_\kappa^2}\right)<\log\left(\frac{1}{0.99^2 \cdot 2 \sqrt{x}}\right) \leq \log\left(\frac{1}{0.99^2 \cdot 2 \sqrt{\kappa/12}}\right)<\log\left(\frac{2}{\sqrt{\kappa}}\right).
\end{align*}
Therefore, when $x\in [\kappa/12,0.04]$, we have $F_\kappa(x)< F_{\kappa}(1/2)$. Combining all the steps, we have the lemma.
\end{proof}

\begin{lem}\label{l:hypoboundonk}
Define a function
\begin{align*}
    \mathcal{L}_1(\kappa):=\frac{\pi P_\kappa^2}{\alpha_c(\kappa)2 \kappa^2 \exp(-\kappa^2)}.
\end{align*}
Then,
\begin{align*}
    \mathcal{L}_1(\kappa) \geq 1.01 , \text{ if } \kappa \in [0.41,2.2].
\end{align*}
Moreover, we have a more detailed description:
\begin{align*}
    &\mathcal{L}_1(\kappa) \geq 1.02 , \text{ if } \kappa \in [1.18,1.23]. \quad &&\mathcal{L}_1(\kappa) \geq 1.031 , \text{ if } \kappa \in [1.13,1.18].\\
    &\mathcal{L}_1(\kappa) \geq 1.041 , \text{ if } \kappa \in [1.08,1.13]. \quad
    &&\mathcal{L}_1(\kappa) \geq 1.06 , \text{ if } \kappa \in [1,1.08].\\
    &\mathcal{L}_1(\kappa) \geq 1.09 , \text{ if } \kappa \in [0.86,1]. \quad &&\mathcal{L}_1(\kappa) \geq 1.17 , \text{ if } \kappa \in [0.6,0.86].\\
    &\mathcal{L}_1(\kappa) \geq 1.45 , \text{ if } \kappa \in [0.41,0.6]\cup [2,2.2].
\end{align*}
\end{lem}

With this lemma, we are able to establish the following lemma.
\begin{lem}\label{l:hypostep2.2}
For any $\kappa\in [0.1,4]$, $F_{\kappa, \alpha_c}(x)< F_{\kappa, \alpha_c}(1/2)$ for any $x\in [0.3,1/2]$.
\end{lem}
\begin{proof}[Proof of Lemma \ref{l:hypostep2.2}]
Recall that if we can show
\begin{align*}
    \mathcal{L}(x)=\frac{\alpha_c}{\pi} \frac{\exp(-\frac{\kappa^2}{2(1-x)})-\exp(-\frac{\kappa^2}{2x})}{ \sqrt{x(1-x)}(\log(1-x)-\log(x))}<q_\kappa(x),
\end{align*}
for any $x\in [0.3,0.5)$, then the derivative of $F(x)$ is positive and the lemma follows. Define
\begin{align*}
    \mathcal{E}_\kappa(x)=\frac{\exp(-\frac{\kappa^2}{2(1-x)})-\exp(-\frac{\kappa^2}{2x})}{(\frac{1}{2}-x)4 \exp(-\kappa^2)\kappa^2}.
\end{align*}
Then by Taylor expansion and simple analysis, we have for any $x\in [0.3,0.5]$,
\begin{align*}
    &\mathcal{E}_\kappa(x) \leq 1-0.5(x-0.5)^2 , \text{ if } \kappa \in [1.23,2]. \  &&\mathcal{E}_\kappa(x) \leq 1-0.25(x-0.5)^2 , \text{ if } \kappa \in [1.18,1.23].\\
    &\mathcal{E}_\kappa(x) \leq 1 , \text{ if } \kappa \in [1.13,1.18]. \ 
    &&\mathcal{E}_\kappa(x) \leq 1+0.25(x-0.5)^2 , \text{ if } \kappa \in [1.08,1.13].\\
    &\mathcal{E}_\kappa(x) \leq 1+0.7(x-0.5)^2 , \text{ if } \kappa \in [1,1.08]. \  &&\mathcal{E}_\kappa(x) \leq 1+1.45(x-0.5)^2 , \text{ if } \kappa \in [0.86,1].\\
    &\mathcal{E}_\kappa(x) \leq 1+3.35(x-0.5)^2 , \text{ if } \kappa \in [0.6,0.86]. \  && \mathcal{E}_\kappa(x) \leq 1+5(x-0.5)^2 , \text{ if } \kappa \in [0.41,0.6]\cup[2,2.2].
\end{align*}
Note that by Lemma \ref{l:hypoboundonk}, we have explicit bounds on $\mathcal{L}_1(\kappa)=\mathcal{L}(0.5)/q_\kappa(0.5)$. Together with the fact that the function
\begin{align*}
    \frac{(1+c(x-0.5)^2)(0.5-x)}{\sqrt{x(1-x)}(\log(1-x)-\log(x))}
\end{align*}
is decreasing from $[0.3,0.5]$, when $c\geq-0.5$, we have the lemma for $\kappa \in [0.41,2.2]$. It remains to check the case when $\kappa \in [0.1,0.41]$ and $\kappa \in [2.2,4]$. We will check that for any $\kappa\in [2.2,4]$ and $x\in [0.3,0.5]$,
\begin{align*}
    \mathcal{L}(x)=\frac{\alpha_c}{\pi} \frac{\exp(-\frac{\kappa^2}{2(1-x)})-\exp(-\frac{\kappa^2}{2x})}{ \sqrt{x(1-x)}(\log(1-x)-\log(x))}<P_\kappa^2.
\end{align*}
In this case, $\mathcal{L}(x)$ is decreasing for $x\in [0.3,0.5]$, therefore, we only need to check 
\begin{align*}
    \mathcal{L}(0.3)=\frac{\alpha_c}{\pi} \frac{\exp(-\frac{\kappa^2}{2(1-0.3)})-\exp(-\frac{\kappa^2}{2\cdot 0.3})}{ \sqrt{0.3(1-0.3)}(\log(1-0.3)-\log(0.3))}<P_\kappa^2.
\end{align*}
Note that we have
\begin{align*}
    \mathcal{L}(0.3)(-\log(P_\kappa))<\frac{\log(2)}{\pi} \frac{\exp(-\frac{\kappa^2}{1.4})}{0.388}<0.569 \exp(-\frac{\kappa^2}{1.4}).
\end{align*}
At the same time, as 
\begin{align*}
    P_x=\erf\left(\frac{x}{\sqrt{2}}\right)\leq 1+e^{-x^2/2}\sqrt{\frac{2}{\pi}}\left(-\frac{1}{x}+\frac{1}{x^3}\right),
\end{align*}
$-x^2\log(x)\geq -(x-1)-3/2(x-1)^2$ when $x<1$ and $x-3/2x^2$ is increasing when $x<0.2$, we have that for $\kappa \geq 2.2$,
\begin{align*}
    (-\log(P_\kappa))P_\kappa^2\geq -(P_\kappa-1)-3/2(P_\kappa-1)^2 \geq e^{-\kappa^2/2}\sqrt{\frac{2}{\pi}}\left(\frac{1}{\kappa}-\frac{1}{\kappa^3}\right)-\frac{3}{2}e^{-\kappa^2}\frac{2}{\pi}\left(\frac{1}{\kappa}-\frac{1}{\kappa^3}\right)^2.
\end{align*}
Then, as the following is increasing in $\kappa$, we have
\begin{align*}
    &\exp\left(\frac{\kappa^2}{1.4}\right)\left( e^{-\kappa^2/2}\sqrt{\frac{2}{\pi}}\left(\frac{1}{\kappa}-\frac{1}{\kappa^3}\right)-\frac{3}{2}e^{-\kappa^2}\frac{2}{\pi}\left(\frac{1}{\kappa}-\frac{1}{\kappa^3}\right)^2 \right)\\
    & \geq \exp\left(\frac{2^2}{1.4}\right)\left( e^{-2^2/2}\sqrt{\frac{2}{\pi}}\left(\frac{1}{2}-\frac{1}{2^3}\right)-\frac{3}{2}e^{-2^2}\frac{2}{\pi}\left(\frac{1}{2}-\frac{1}{2^3}\right)^2 \right)>0.6.
\end{align*}
This shows that for any $x\in [0.3,0.5]$ and $\kappa \in [2.2,4]$, $\mathcal{L}(x)<q_\kappa(x)$.
Now we deal with the case when $\kappa$ is small. Note that for any $\kappa<0.41$, we have that $\mathcal{L}(x)$ is decreasing for $x\in [0.3,0.5]$. Therefore, 
\begin{align*}
    \mathcal{L}(x)\leq \mathcal{L}(0.3)=\frac{\alpha_c}{\pi} \frac{\exp(-\frac{\kappa^2}{2(1-0.3)})-\exp(-\frac{\kappa^2}{2\cdot 0.3})}{ \sqrt{0.3(1-0.3)}(\log(1-0.3)-\log(0.3))}<\frac{\alpha_c}{\pi} \frac{-\frac{\kappa^2}{2(1-0.3)}+\frac{\kappa^2}{2\cdot 0.3}}{0.388}.
\end{align*}
Note that when $\kappa\leq 0.41$,
\begin{align*}
    \erf(\frac{\kappa}{\sqrt{2}})=\frac{1}{\sqrt{2\pi}} \int_{-\kappa}^\kappa \exp(-x^2/2) dx \geq 0.92\frac{2\kappa}{\sqrt{2\pi}}.
\end{align*}
Further, as $-x^2\log(x)$ is increasing when $x< 0.5$, we have that for any $\kappa\in [0.1,0.41]$,
\begin{align*}
    -\log(P_\kappa)P_\kappa^2\kappa^{-2}\geq -\frac{0.92^2\cdot 2}{\pi} \log(\frac{0.92\cdot 2\kappa}{\sqrt{2\pi }})&\geq -\frac{0.92^2\cdot 2}{\pi} \log(\frac{0.92\cdot 2\cdot 0.4}{\sqrt{2\pi }}) \geq 0.66\\
    &>0.55> \frac{\log(2)}{\pi} \frac{-\frac{1}{2(1-0.3)}+\frac{1}{2\cdot 0.3}}{0.388}.
\end{align*}
Therefore, the lemma holds for $\kappa\in [0.1,0.41]$.
\end{proof}

\begin{proof}[Proof of Lemma \ref{l:hypoboundonk}]
Note that
\begin{align*}
    \mathcal{L}_1(\kappa)=\frac{\pi P_\kappa^2}{\alpha_c(\kappa)2 \kappa^2 \exp(-\kappa^2)}= -\frac{\pi}{2\log(2)}\frac{P_\kappa^2 \log(P_\kappa)}{\exp(-\kappa^2)\kappa^{2}}.
\end{align*}
As we have constructed explicit bounds of $\erf(x/\sqrt{2})$, $\exp(-x^2/2)$ and $\log(x)$, see Section $\ref{sss:specialfn}$ for detailed definitions, we have
\begin{align*}
    \mathcal{L}_1(\kappa)\geq -2.26618 \frac{h_l^2(\kappa)\mathcal{G}_u(h_u(\kappa))}{g_u^2(\kappa) \kappa^2}.
\end{align*}
We check that for any $\kappa \in [[0.41,2,0.002]]$, the right hand side satisfies the inequalities required by the lemma by at least $0.002$. Here we use $[[a,b,\delta]]$ to denote the grid points that starts at $a$, ends at $b$ with step size $\delta$.

To see that this implies the lemma, we compute the second derivative of $\mathcal{L}_1(x)$.
\begin{align*}
    \mathcal{L}_1''(x)&=-\frac{2 \left(2 \log
   \left(\erf\left(\frac{x}{\sqrt{2}}\right)\right)+3\right)}{x^2 \log (4)}-\frac{\sqrt{2 \pi }
   e^{\frac{x^2}{2}} \left(3 x^2-4\right)  \erf\left(\frac{x}{\sqrt{2}}\right)
   \left(2 \log \left(\erf\left(\frac{x}{\sqrt{2}}\right)\right)+1\right)}{x^3 \log (4)}\\
   &\quad -\frac{2 \pi 
   e^{x^2} \left(2 x^4-3 x^2+3\right) \erf\left(\frac{x}{\sqrt{2}}\right)^2 \log
   \left(\erf\left(\frac{x}{\sqrt{2}}\right)\right)}{x^4 \log (4)}\\
   &:=A_1+A_2+A_3.
\end{align*}
We consider all $x\in [0.41,2.2]$. Note that $A_1$ is always negative when $x\in [0.41, 2.2]$. $A_2$ is positive only when $3x^2-4< 0$ and $2\log(\erf(x/\sqrt{2}))+1>0$. In this case, $2\log(\erf(x/\sqrt{2}))+1<0.5$. Therefore, we have that
\begin{align*}
    A_2\leq -\frac{\sqrt{2 \pi }
   e^{\frac{0.41^2}{2}} \left(3 \cdot 0.41^2-4\right)}{2\cdot 0.41^3 \log (4)}\leq 50.
\end{align*}
For $A_3$, notice that $\exp(x^2)(2x^4-3x^2+3)/x^4$ is convex and thus it achieves maximum at $x=2.2$. Moreover, $-\erf(\frac{x}{\sqrt{2}})^2 \log(\erf(\frac{x}{\sqrt{2}})$ is bounded by $e^{-1/2}$. This implies that 
\begin{align*}
    A_3\leq \frac{2.2 \pi 
   e^{2.2^2} \left(2 \cdot 2.2^4-3 \cdot 2.2^2+3\right) e^{-1/2}}{2.2^4 \log (4)}\leq 550.
\end{align*}
Combining the computations above, we have that
\begin{align*}
    \mathcal{L}_1''(x)\leq 600. 
\end{align*}
Therefore, the lemma is established.
\end{proof}

\begin{lem}\label{l:hypostep2.3}
For any $\kappa\in [0.1,4]$, $F_{\kappa, \alpha_c}(x)< F_{\kappa, \alpha_c}(1/2)$ for any $x\in [0,0.005]$.
\end{lem}
\begin{proof}[Proof of Lemma \ref{l:hypostep2.3}]
We will check that for any $x\in [0,0.005]$,
\begin{align*}
    \mathcal{L}(x)=\frac{\alpha_c}{\pi} \frac{\exp(-\frac{\kappa^2}{2(1-x)})-\exp(-\frac{\kappa^2}{2x})}{ \sqrt{x(1-x)}(\log(1-x)-\log(x))}>P_\kappa.
\end{align*}
Note that $\mathcal{L}(x)$ is decreasing for $x\in [0,0.005]$. Therefore, we only need to show the inequality for $x=0.005$. This is equivalent to 
\begin{align*}
    \frac{\exp(-\frac{\kappa^2}{2(1-0.005)})-\exp(-\frac{\kappa^2}{2\cdot 0.005})}{-\log(P_\kappa)P_\kappa}>\frac{\pi}{ \log(2)}\sqrt{0.995\cdot 0.005}(\log(0.995)-\log(0.005)).
\end{align*}
We use $L_0(\kappa)$ to denote the left hand side. The right hand side is smaller than $1.7$. So it only remains to show that $L_0(\kappa)>1.7$. As $-x\log(x)$ is maximized at $1/e$ with max value $1/e$, for $\kappa\in [0.1,0.95]$, 
\begin{align*}
    L_0(\kappa)\geq e (\exp(-\kappa^2/(2(1-0.005)))-\exp(-\kappa^2/(2\cdot 0.005)))> 1.7.
\end{align*}
Next, note that
\begin{align*}
    1-\erf\left(\frac{\kappa}{\sqrt{2}}\right)\leq \exp\left(-\frac{\kappa^2}{2}\right)\sqrt{\frac{2}{\pi}}\frac{1}{\kappa}.
\end{align*}
As $-x\log(x)\leq 1-x$, we have that when $\kappa \in [1.45,4]$, 
\begin{align*}
    L_0(\kappa)& \geq (\exp(-\kappa^2/(2(1-0.005)))-\exp(-\kappa^2/(2\cdot 0.005)))\exp\left(\frac{\kappa^2}{2}\right)\sqrt{\frac{\pi}{2}}\kappa \\
    &\geq (\exp\left(\frac{-0.005\kappa^2}{2\cdot 0.995}\right)-\exp(-99.5\kappa^2))\sqrt{\frac{\pi}{2}}\kappa >1.7.
\end{align*}
Further, when $\kappa \in (0.95,1.2)$, 
\begin{align*}
    L_0(\kappa) \geq\frac{\exp(-1.2^2/(2(1-0.005)))-\exp(-0.95^2/(2\cdot 0.005))}{-\log(P_{0.95})P_{0.95}}>1.7.
\end{align*}
Similarly, when $\kappa \in (1.2,1.45)$,
\begin{align*}
    L_0(\kappa) \geq\frac{\exp(-1.45^2/(2(1-0.005)))-\exp(-1.2^2/(2\cdot 0.005))}{-\log(P_{1.2})P_{1.2}}>1.7.
\end{align*}
Therefore, we have the lemma.
\end{proof}

\begin{lem}\label{l:hypostep3}
For any $\kappa\in [0.1,4]$, $F_{\kappa, \alpha_c}(x)< F_{\kappa, \alpha_c}(1/2)$ for any $x\in [0.005,0.3]$.
\end{lem}
\begin{proof}
We will show that for any $\kappa \in [0.1,4]$ and $x\in [0.005,0.3]$,
\begin{align*}
    F_{\kappa,\alpha_c}(x)-F_{\kappa,\alpha_c}(1/2)=\alpha_c(\kappa) \log q_\kappa(x)-x \log x -(1-x) \log (1-x) - \alpha_c(\kappa) \log P_\kappa <0.
\end{align*}
We check that on all grid points $[[0.005,0.3,0.0005]]_x \times [[0.1,4,0.0005]]_\kappa$,
\begin{align*}
    \frac{\log(2)}{-\mathcal{G}_u(g_u(\kappa))} \mathcal{G}_u( q_{\kappa,u}(x))-x \mathcal{G}_l(x) -(1-x) \mathcal{G}_l(1-x) -0.69314718 <-0.006.
\end{align*}
This implies that 
\begin{align*}
    F_{\kappa,\alpha_c}(x)-F_{\kappa,\alpha_c}(1/2)<-0.006,
\end{align*}
on any grid points $[[0.005,0.3,0.0005]]_x \times [[0.1,4,0.0005]]_\kappa$. Together with the following bounds on the derivative, we have the lemma. Note that
\begin{align*}
    \frac{\partial}{\partial x}(F_{\kappa,\alpha_c}(x)-F_{\kappa,\alpha_c}(1/2)) &= -\log(x)+\log(1-x)+\alpha_c(\kappa) \frac{-\exp(-\frac{\kappa^2}{2(1-x)})+\exp(-\frac{\kappa^2}{2x})}{\pi \sqrt{x(1-x)}q_\kappa(x)}\\
    &\leq -\log(0.005)+\log(1-0.005)<6.
\end{align*}
And
\begin{align*}
    \frac{\partial}{\partial \kappa}(F_{\kappa,\alpha_c}(x)-F_{\kappa,\alpha_c}(1/2))&=\frac{\partial}{\partial \kappa}\left(\frac{\log(2)\log(q_\kappa(x))}{-\log(P_\kappa)}\right)\geq \frac{\log(2)\log(q_\kappa(x))\exp(-\frac{\kappa^2}{2})\sqrt{\frac{2}{\pi}}}{P_\kappa \log(P_\kappa)^2}\\
    &\geq \frac{2\log(2)\exp(-\frac{\kappa^2}{2})\sqrt{\frac{2}{\pi}}}{P_\kappa \log(P_\kappa)}.
\end{align*}
Note that for any $\kappa\in [0.1,1]$, $-P_\kappa \log(P_\kappa)>0.2 $. This implies that for $\kappa\in [0.1,1]$,
\begin{align*}
    \frac{\partial}{\partial \kappa}(F_{\kappa,\alpha_c}(x)-F_{\kappa,\alpha_c}(1/2))\geq \frac{-2\log(2)\sqrt{\frac{2}{\pi}}}{0.2}>-6.
\end{align*}
For $\kappa\in [1,1.5]$, as $-P_\kappa \log(P_\kappa)$ is decreasing,
\begin{align*}
    \frac{\partial}{\partial \kappa}(F_{\kappa,\alpha_c}(x)-F_{\kappa,\alpha_c}(1/2))\geq \frac{2\log(2)\exp(-\frac{1^2}{2})\sqrt{\frac{2}{\pi}}}{P_{1.5} \log(P_{1.5})} > -6.
\end{align*}
For $\kappa\in [1.5,4]$, note that as 
\begin{align*}
    P_\kappa=\erf\left(\frac{\kappa}{\sqrt{2}}\right)\leq 1+\exp\left(-\frac{\kappa^2}{2}\right)\sqrt{\frac{2}{\pi}}\left(-\frac{1}{\kappa}+\frac{1}{\kappa^3}\right),
\end{align*}
and $-x\log(x)\geq (1-x)-(1-x)^2/2$, we have that
\begin{align*}
    \frac{\partial}{\partial \kappa}(F_{\kappa,\alpha_c}(x)-F_{\kappa,\alpha_c}(1/2)) &\geq -\frac{2\log(2)\exp(-\frac{\kappa^2}{2})\sqrt{\frac{2}{\pi}}}{\exp(-\frac{\kappa^2}{2})\sqrt{\frac{2}{\pi}}(\frac{1}{\kappa}-\frac{1}{\kappa^3})-\frac{1}{2}\exp(-\kappa^2)\frac{2}{\pi}(\frac{1}{\kappa}-\frac{1}{\kappa^3})^2}\\
    &= -\frac{2\log(2)}{\frac{1}{\kappa}-\frac{1}{\kappa^3}-\frac{1}{2}\exp(-\frac{\kappa^2}{2})\sqrt{\frac{2}{\pi}}(\frac{1}{\kappa}-\frac{1}{\kappa^3})^2}.
\end{align*}
It is not hard to see that the denominator is minimized at $\kappa=4$, when $\kappa \in [1.5,4]$. This implies that
\begin{align*}
    \frac{\partial}{\partial \kappa}(F_{\kappa,\alpha_c}(x)-F_{\kappa,\alpha_c}(1/2)) \geq -\frac{2\log(2)}{\frac{1}{4}-\frac{1}{4^3}-\frac{1}{2}\exp(-\frac{4^2}{2})\sqrt{\frac{2}{\pi}}(\frac{1}{4}-\frac{1}{4^3})^2} >-6.
\end{align*}
\end{proof}

\subsubsection{Bounds on Special Functions}\label{sss:specialfn}
\textbf{Polynomial bound on $\erf(\frac{x}{\sqrt{2}})$.} We will use $g_u$ and $g_l$ to denote the upper and lower bound respectively. By Taylor expansion, we know that for any $x\geq 0$,
\begin{align*}
    \erf\left(\frac{x}{ \sqrt{2}}\right)&\geq \sqrt{\frac{2}{\pi }} x-\frac{x^3}{3 \sqrt{2 \pi }}+\frac{x^5}{20 \sqrt{2 \pi
   }}-\frac{x^7}{168 \sqrt{2 \pi }}\\
   &\geq 0.797884 x-0.132982 x^3+0.0199470 x^5-0.00237467 x^7:=g_{l,0}(x).
\end{align*}
By Taylor expansion at $1$,
\begin{align*}
    \erf\left(\frac{x}{ \sqrt{2}}\right)
    &\geq \erf\left(\frac{1}{\sqrt{2}}\right)+\sqrt{\frac{2}{e \pi }}
   (x-1)-\frac{(x-1)^2}{\sqrt{2 e \pi }}+\frac{(x-1)^4}{6 \sqrt{2 e \pi
   }}\\
   & \quad -\frac{(x-1)^5}{30 \sqrt{2 e \pi }}-\frac{(x-1)^6}{60 \sqrt{2 e \pi }}+\frac{1}{315}
   \sqrt{\frac{2}{e \pi }} (x-1)^7\\
   &\geq -0.000397705 + 0.801191 x - 0.0120986 x^2 - 0.107544 x^3 - 
 0.0336071 x^4 \\
 & \quad + 0.0483940 x^5 - 0.0147872 x^6 + 0.00153631 x^7:=g_{l,1}(x).
\end{align*}
By Taylor expansion at $2$,
\begin{align*}
   \erf\left(\frac{x}{ \sqrt{2}}\right)
    &\geq \erf\left(\sqrt{2}\right)+\frac{\sqrt{\frac{2}{\pi }}
   (x-2)}{e^2}-\frac{\sqrt{\frac{2}{\pi }} (x-2)^2}{e^2}+\frac{(x-2)^3}{e^2 \sqrt{2 \pi
   }}-\frac{(x-2)^4}{6 \left(e^2 \sqrt{2 \pi }\right)}-\frac{(x-2)^5}{12 \left(e^2
   \sqrt{2 \pi }\right)}\\
   &\geq -0.125320 + 1.11581 x - 0.287952 x^2 - 0.0539910 x^3 \\
  &\quad  +  0.0359939 x^4 - 0.00449925 x^5:=g_{l,2}(x).
\end{align*}
By Taylor expansion at $3$,
\begin{align*}
\erf\left(\frac{x}{ \sqrt{2}}\right)
    &\geq 
    \text{erf}\left(\frac{3}{\sqrt{2}}\right)+\frac{\sqrt{\frac{2}{\pi }}
   (x-3)}{e^{9/2}}-\frac{3 (x-3)^2}{e^{9/2} \sqrt{2 \pi }}+\frac{4 \sqrt{\frac{2}{\pi }}
   (x-3)^3}{3 e^{9/2}}-\frac{3 (x-3)^4}{2 \left(e^{9/2} \sqrt{2 \pi
   }\right)}\\
   & \quad +\frac{(x-3)^5}{2 e^{9/2} \sqrt{2 \pi }}-\frac{(x-3)^6}{20 \left(e^{9/2}
   \sqrt{2 \pi }\right)}-\frac{2 \sqrt{\frac{2}{\pi }} (x-3)^7}{105 e^{9/2}}\\
   &\geq -0.337289 + 1.48466 x - 0.484624 x^2 - 0.0679551 x^3 + 
 0.0897448 x^4 \\
 & \quad - 0.0257048 x^5 + 0.00332388 x^6 - 0.000168833 x^7:=g_{l,3}(x).
\end{align*}
By Taylor expansion at $4$,
\begin{align*}
\erf\left(\frac{x}{ \sqrt{2}}\right)
    &\geq 
    \text{erf}\left(2 \sqrt{2}\right)+\frac{\sqrt{\frac{2}{\pi }} (x-4)}{e^8}-\frac{2
   \sqrt{\frac{2}{\pi }} (x-4)^2}{e^8}+\frac{5 (x-4)^3}{e^8 \sqrt{2 \pi }}-\frac{13
   (x-4)^4}{3 \left(e^8 \sqrt{2 \pi }\right)}+\frac{163 (x-4)^5}{60 e^8 \sqrt{2 \pi }} \\
   & \quad -\frac{37 (x-4)^6}{30 \left(e^8 \sqrt{2 \pi }\right)}+\frac{961 (x-4)^7}{2520 e^8
   \sqrt{2 \pi }}-\frac{59 (x-4)^8}{1008 \left(e^8 \sqrt{2 \pi }\right)}-\frac{223
   (x-4)^9}{20160 \left(e^8 \sqrt{2 \pi }\right)}\\
   &\geq -1.21084 + 3.28149 x - 2.05347 x^2 + 0.676529 x^3 - 0.111157 x^4 + 
 0.00179703 x^5 \\
 &\quad + 0.00285503 x^6 - 0.000550989 x^7 + 
 0.0000454597 x^8 - 0.00000148037 x^9:=g_{l,4}(x).
\end{align*}
Define a function 
\begin{align*}
    g_l(x):=
    \begin{cases}
    g_{l,0}(x), & \quad x\in [0,0.5)\\
    g_{l,1}(x), & \quad x\in [0.5,1.5)\\
    g_{l,2}(x), & \quad x\in [1.5,2.5)\\
    g_{l,3}(x), & \quad x\in [2.5,3.5)\\
    g_{l,4}(x), & \quad x\in [3.5,4.5)\\
    0.99999, & \quad x\in [4.5,\infty)
    \end{cases}
\end{align*}
Therefore $g_l$ is a lower bound of $\erf(x/\sqrt{2})$ when $x\geq 0$. We give an upper bound similarly. By Taylor expansion, we know that for any $x\geq 0$,
\begin{align*}
    \erf\left(\frac{x}{ \sqrt{2}}\right)&\leq \sqrt{\frac{2}{\pi }} x-\frac{x^3}{3 \sqrt{2 \pi }}+\frac{x^5}{20 \sqrt{2 \pi
   }}\\
   &\leq 0.797886 x - 0.132980 x^3 + 0.0199472 x^5:=g_{u,0}(x).
\end{align*}
By Taylor expansion at $1$,
\begin{align*}
    \erf\left(\frac{x}{ \sqrt{2}}\right)
    &\leq \erf\left(\frac{1}{\sqrt{2}}\right)+\sqrt{\frac{2}{e \pi }}
   (x-1)-\frac{(x-1)^2}{\sqrt{2 e \pi }}+\frac{(x-1)^4}{6 \sqrt{2 e \pi
   }}-\frac{(x-1)^5}{30 \sqrt{2 e \pi }}\\
   &\leq 0.00517147 + 0.766242 x + 0.0806570 x^2 - 0.241970 x^3 \\
   &\quad + 0.0806570 x^4 - 0.00806568 x^5:=g_{u,1}(x).
\end{align*}
By Taylor expansion at $2$,
\begin{align*}
    \erf\left(\frac{x}{ \sqrt{2}}\right)
    &\leq \erf\left(\sqrt{2}\right)+\frac{\sqrt{\frac{2}{\pi }}
   (x-2)}{e^2}-\frac{\sqrt{\frac{2}{\pi }} (x-2)^2}{e^2}+\frac{(x-2)^3}{e^2 \sqrt{2 \pi
   }}-\frac{(x-2)^4}{6 \left(e^2 \sqrt{2 \pi }\right)}\\
   & \quad -\frac{(x-2)^5}{12 \left(e^2
   \sqrt{2 \pi }\right)}+\frac{(x-2)^6}{20 e^2 \sqrt{2 \pi }}-\frac{11 (x-2)^7}{2520
   \left(e^2 \sqrt{2 \pi }\right)}\\
   &\leq 0.0776180 + 0.491917 x + 0.518314 x^2 - 0.617896 x^3 + 0.263957 x^4 \\
   &\quad - 
 0.0566904 x^5 + 0.00599901 x^6 - 0.000235674 x^7:=g_{u,2}(x).
\end{align*}
By Taylor expansion at $3$,
\begin{align*}
    \erf\left(\frac{x}{ \sqrt{2}}\right)
    &\leq \erf\left(\frac{3}{\sqrt{2}}\right)+\frac{\sqrt{\frac{2}{\pi }}
   (x-3)}{e^{9/2}}-\frac{3 (x-3)^2}{e^{9/2} \sqrt{2 \pi }}+\frac{4 \sqrt{\frac{2}{\pi }}
   (x-3)^3}{3 e^{9/2}}-\frac{3 (x-3)^4}{2 \left(e^{9/2} \sqrt{2 \pi
   }\right)}+\frac{(x-3)^5}{2 e^{9/2} \sqrt{2 \pi }}\\
   &\leq -0.544982 + 2.02315 x - 1.07693 x^2 + 0.291026 x^3 \\
   & \quad - 0.0398865 x^4 + 
 0.00221593 x^5:=g_{u,3}(x).
\end{align*}
By Taylor expansion at $4$,
\begin{align*}
    \erf\left(\frac{x}{ \sqrt{2}}\right)
    &\leq \erf\left(2 \sqrt{2}\right)+\frac{\sqrt{\frac{2}{\pi }} (x-4)}{e^8}-\frac{2
   \sqrt{\frac{2}{\pi }} (x-4)^2}{e^8}+\frac{5 (x-4)^3}{e^8 \sqrt{2 \pi }}-\frac{13
   (x-4)^4}{3 \left(e^8 \sqrt{2 \pi }\right)}+\frac{163 (x-4)^5}{60 e^8 \sqrt{2 \pi
   }}\\
   &\leq 0.426716 + 0.650505 x - 0.296924 x^2 + 0.0681197 x^3 \\
   & \quad - 0.00785136 x^4 + 0.000363573 x^5:=g_{u,4}(x).
\end{align*}
Define a function 
\begin{align*}
    g_u(x):=
    \begin{cases}
    g_{u,0}(x), & \quad x\in [0,0.5)\\
    g_{u,1}(x), & \quad x\in [0.5,1.5)\\
    g_{u,2}(x), & \quad x\in [1.5,2.5)\\
    g_{u,3}(x), & \quad x\in [2.5,3.5)\\
    g_{u,4}(x), & \quad x\in [3.5,4.5)\\
    1, & \quad x\in [4.5,\infty)
    \end{cases}
\end{align*}
Therefore $g_u$ is an upper bound of $\erf(x/\sqrt{2})$ when $x\geq 0$. Note that both $g_l$ and $g_u$ are polynomials with rational coefficient. We can compute their values exactly on rational points.\\

\noindent\textbf{Polynomial bound on $\exp(-x^2/2)$.} Similarly, we give lower and upper bounds on $\exp(-x^2/2)$. We use $h_l$ and $h_u$ to denote them respectively. By Taylor expansion, we know that
\begin{align*}
    \exp\left(\frac{-x^2}{2}\right) \geq 1-\frac{x^2}{2}+\frac{x^4}{8}-\frac{x^6}{48}:=h_{l,0}(x).
\end{align*}
By Taylor expansion at $1$,
\begin{align*}
\exp\left(\frac{-x^2}{2}\right) &\geq
    \frac{1}{\sqrt{e}}-\frac{x-1}{\sqrt{e}}+\frac{(x-1)^3}{3 \sqrt{e}}-\frac{(x-1)^4}{12
   \sqrt{e}}-\frac{(x-1)^5}{20 \sqrt{e}}\\
   &\geq 0.990666 + 0.0505441 x - 0.606532 x^2 + 0.101087 x^3 \\
   & \quad + 0.101087 x^4 - 
 0.0303266 x^5 :=h_{l,1}(x).
\end{align*}
By Taylor expansion at $2$,
\begin{align*}
\exp\left(\frac{-x^2}{2}\right) &\geq
    \frac{1}{e^2}-\frac{2 (x-2)}{e^2}+\frac{3 (x-2)^2}{2 e^2}-\frac{(x-2)^3}{3 e^2}-\frac{5
   (x-2)^4}{24 e^2}\\
   & \quad +\frac{3 (x-2)^5}{20 e^2}-\frac{11 (x-2)^6}{720 e^2}-\frac{43
   (x-2)^7}{2520 e^2}\\
   &\geq 0.912116 + 0.264655 x - 0.771412 x^2 + 0.0300744 x^3 + 0.291346 x^4 \\
   & \quad - 0.148870 x^5 + 0.0302624 x^6 - 0.00230930 x^7 :=h_{l,2}(x).
\end{align*}
By Taylor expansion at $3$,
\begin{align*}
\exp\left(\frac{-x^2}{2}\right) &\geq
    \frac{1}{e^{9/2}}-\frac{3 (x-3)}{e^{9/2}}+\frac{4 (x-3)^2}{e^{9/2}}-\frac{3
   (x-3)^3}{e^{9/2}}+\frac{5 (x-3)^4}{4 e^{9/2}}-\frac{3 (x-3)^5}{20 e^{9/2}}\\
   & \quad -\frac{2
   (x-3)^6}{15 e^{9/2}}+\frac{11 (x-3)^7}{140 e^{9/2}}-\frac{43 (x-3)^8}{3360
   e^{9/2}}-\frac{(x-3)^9}{224 e^{9/2}}\\
   &\geq -0.00478085 + 2.79830 x - 3.70700 x^2 + 1.82214 x^3 - 0.273560 x^4 - 
 0.101232 x^5 \\
 & \quad + 0.0568409 x^6 - 0.0117836 x^7 + 0.00119685 x^8 - 
 0.0000495938 x^9 :=h_{l,3}(x).
\end{align*}
By Taylor expansion at $4$,
\begin{align*}
\exp\left(\frac{-x^2}{2}\right) &\geq
    \frac{1}{e^8}-\frac{4 (x-4)}{e^8}+\frac{15 (x-4)^2}{2 e^8}-\frac{26 (x-4)^3}{3
   e^8}+\frac{163 (x-4)^4}{24 e^8}-\frac{37 (x-4)^5}{10 e^8}\\
   &\geq 2.08628 - 2.33304 x + 1.05049 x^2 - 0.237956 x^3 \\
   & \quad + 0.0271025 x^4 - 
 0.00124122 x^5 :=h_{l,4}(x).
\end{align*}
Define a function 
\begin{align*}
    h_l(x):=
    \begin{cases}
    h_{l,0}(x), & \quad x\in [0,0.5)\\
    h_{l,1}(x), & \quad x\in [0.5,1.5)\\
    h_{l,2}(x), & \quad x\in [1.5,2.5)\\
    h_{l,3}(x), & \quad x\in [2.5,3.5)\\
    h_{l,4}(x), & \quad x\in [3.5,4.5)\\
    0, & \quad x\in [4.5,\infty)
    \end{cases}
\end{align*}
We further define $h_l(x):=h_l(-x)$ when $x<0$. Then $h_l$ is a lower bound of $\exp(-x^2/2)$. Similar, we construct an upper bound of $\exp(-x^2/2)$. By Taylor expansion, we know that
\begin{align*}
    \exp\left(\frac{-x^2}{2}\right) \leq 1-\frac{x^2}{2}+\frac{x^4}{8}-\frac{x^6}{48}+\frac{x^8}{384}:=h_{u,0}(x).
\end{align*}
By Taylor expansion at $1$,
\begin{align*}
\exp\left(\frac{-x^2}{2}\right) &\leq
    \frac{1}{\sqrt{e}}-\frac{x-1}{\sqrt{e}}+\frac{(x-1)^3}{3 \sqrt{e}}-\frac{(x-1)^4}{12
   \sqrt{e}}-\frac{(x-1)^5}{20 \sqrt{e}}+\frac{(x-1)^6}{45 \sqrt{e}}+\frac{(x-1)^7}{252
   \sqrt{e}}\\
   &\leq 1.00175 - 0.0134784 x - 0.454897 x^2 - 0.0842403 x^3 + 0.219026 x^4 \\
   & \quad -  0.0606530 x^5 - 0.00336960 x^6 + 0.00240688 x^7 :=h_{u,1}(x).
\end{align*}
By Taylor expansion at $2$,
\begin{align*}
\exp\left(\frac{-x^2}{2}\right) &\leq
    \frac{1}{e^2}-\frac{2 (x-2)}{e^2}+\frac{3 (x-2)^2}{2 e^2}-\frac{(x-2)^3}{3 e^2}-\frac{5
   (x-2)^4}{24 e^2}+\frac{3 (x-2)^5}{20 e^2}\\
   &\leq 0.748856 + 0.902236 x - 1.82702 x^2 + 0.992460 x^3 \\
   & \quad - 0.231197 x^4 + 
 0.0203004 x^5 :=h_{u,2}(x).
\end{align*}
By Taylor expansion at $3$,
\begin{align*}
\exp\left(\frac{-x^2}{2}\right) &\leq
    \frac{1}{e^{9/2}}-\frac{3 (x-3)}{e^{9/2}}+\frac{4 (x-3)^2}{e^{9/2}}-\frac{3
   (x-3)^3}{e^{9/2}}+\frac{5 (x-3)^4}{4 e^{9/2}}-\frac{3 (x-3)^5}{20
   e^{9/2}}\\
   &\leq 2.94056 - 3.37435 x + 1.54416 x^2 - 0.349932 x^3 \\
   & \quad + 0.0388816 x^4 - 
 0.00166634 x^5 :=h_{u,3}(x).
\end{align*}
By Taylor expansion at $4$,
\begin{align*}
\exp\left(\frac{-x^2}{2}\right) &\leq
    \frac{1}{e^8}-\frac{4 (x-4)}{e^8}+\frac{15 (x-4)^2}{2 e^8}-\frac{26 (x-4)^3}{3
   e^8}+\frac{163 (x-4)^4}{24 e^8}\\
   & \quad -\frac{37 (x-4)^5}{10 e^8}+\frac{961 (x-4)^6}{720
   e^8}-\frac{59 (x-4)^7}{252 e^8}\\
   &\leq 5.20709 - 7.33592 x + 4.45881 x^2 - 1.51479 x^3 + 0.310495 x^4 \\
   & \quad - 0.0383768 x^5 + 0.00264690 x^6 - 0.0000785408 x^7 :=h_{u,4}(x).
\end{align*}
Define a function 
\begin{align*}
    h_u(x):=
    \begin{cases}
    h_{u,0}(x), & \quad x\in [0,0.5)\\
    h_{u,1}(x), & \quad x\in [0.5,1.5)\\
    h_{u,2}(x), & \quad x\in [1.5,2.5)\\
    h_{u,3}(x), & \quad x\in [2.5,3.5)\\
    h_{u,4}(x), & \quad x\in [3.5,4.5)\\
    0.00005, & \quad x\in [4.5,\infty)
    \end{cases}
\end{align*}
We further define $h_u(x):=h_u(-x)$ when $x<0$. Then $h_u$ is an upper bound of $\exp(-x^2/2)$. Note that both $h_l$ and $h_u$ are polynomials with rational coefficient. We can compute their values exactly on rational points.\\

\noindent\textbf{Bound on $\sqrt{x}$.} Note that for any rational $x>0$, we can compute $\sqrt{x}$ with error smaller than $10^{-6}$. Define $\mathcal{S}_l(x):=10^{-6} z$, where $z$ is the maximum in $\{z\in \mathbb{Z},z^2\leq 10^{12}x\}$. In other words, $\mathcal{S}_l(x)$ is a lower bound of $\sqrt{x}$ up to $10^{-6}$. We define similarly $\mathcal{S}_u(x):=10^{-6} z$, where $z$ is the minimum in $\{z\in \mathbb{Z},z^2\geq 10^{12}x\}$.\\

\noindent\textbf{Bound on $q_\kappa(x)$.}
Note that 
\begin{align*}
    q_\kappa(x)
    =\frac{1}{\sqrt{2\pi}} \int_{-\kappa}^{\kappa} \left(\erf\left(\frac{\kappa+(1-2x)y}{2\sqrt{2x(1-x)}}\right)-\erf\left(\frac{-\kappa+(1-2x)y}{2\sqrt{2x(1-x)}}\right)\right)\exp\left(-\frac{y^2}{2}\right) dy.
\end{align*}
If we define
\begin{align*}
    f(x)=\left(\erf\left(\frac{\kappa+(1-2x)y}{2\sqrt{2x(1-x)}}\right)-\erf\left(\frac{-\kappa+(1-2x)y}{2\sqrt{2x(1-x)}}\right)\right)\exp\left(-\frac{y^2}{2}\right),
\end{align*}
then we have
\begin{align*}
    f'(y) & =  e^{-\frac{y^2}{2}} \left(\frac{(1-2 x) e^{-\frac{(\kappa+(1-2 x) y)^2}{8 (1-x) x}}}{\sqrt{2
   \pi } \sqrt{(1-x) x}}+\frac{(2 x-1) e^{-\frac{(\kappa-(1-2 x) y)^2}{8 (1-x) x}}}{\sqrt{2
   \pi } \sqrt{(1-x) x}}\right)\\
   &-e^{-\frac{y^2}{2}} y \left(\erf\left(\frac{\kappa-(1-2
   x) y}{2 \sqrt{2} \sqrt{(1-x) x}}\right)+\erf\left(\frac{\kappa+(1-2 x) y}{2 \sqrt{2}
   \sqrt{(1-x) x}}\right)\right).
\end{align*}
and
\begin{align*}
    &f''(y)=e^{-\frac{y^2}{2}} (y^2-1) \left(\erf\left(\frac{\kappa-(1-2 x) y}{2 \sqrt{2} \sqrt{(1-x)
   x}}\right)+\erf\left(\frac{\kappa+(1-2 x) y}{2 \sqrt{2} \sqrt{(1-x)
   x}}\right)\right)\\
   &-2 e^{-\frac{y^2}{2}} y \left(\frac{(1-2 x)
   e^{-\frac{(\kappa+(1-2 x) y)^2}{8 (1-x) x}}}{\sqrt{2 \pi } \sqrt{(1-x) x}}+\frac{(2 x-1)
   e^{-\frac{(\kappa-(1-2 x) y)^2}{8 (1-x) x}}}{\sqrt{2 \pi } \sqrt{(1-x)
   x}}\right)\\
   &+e^{-\frac{y^2}{2}} \left(-\frac{(1-2 x)^2 e^{-\frac{(\kappa+(1-2 x) y)^2}{8
   (1-x) x}} (\kappa+(1-2 x) y)}{4 \sqrt{2 \pi } (1-x) x \sqrt{(1-x) x}}-\frac{(2 x-1)^2
   e^{-\frac{(\kappa-(1-2 x) y)^2}{8 (1-x) x}} (\kappa-(1-2 x) y)}{4 \sqrt{2 \pi } (1-x) x
   \sqrt{(1-x) x}}\right).
\end{align*}
This implies that for any $x \in [0.005,0.995]$ and $y\in [-\kappa,\kappa]$,
\begin{align*}
    &f''(y)\\
    &\geq -2-\frac{e^{-\frac{y^2}{2}}(1-2 x)^2}{4 \sqrt{2 \pi } (1-x) x \sqrt{(1-x) x}}\left( e^{-\frac{(\kappa+(1-2 x) y)^2}{8
   (1-x) x}} (\kappa+(1-2 x) y) +e^{-\frac{(\kappa-(1-2 x) y)^2}{8 (1-x) x}} (\kappa-(1-2 x) y)\right)\\
   &\geq -2 -\frac{\kappa e^{-\frac{y^2}{2}}(1-2 x)^2}{4 \sqrt{2 \pi } (1-x) x \sqrt{(1-x) x}}\left( e^{-\frac{(\kappa+(1-2 x) y)^2}{8
   (1-x) x}} +e^{-\frac{(\kappa-(1-2 x) y)^2}{8 (1-x) x}}\right)\\
   &\geq -2 - \frac{2e^{-\frac{\kappa^2}{2}}\kappa (1-2 x)^2}{4 \sqrt{2 \pi } (1-x) x \sqrt{(1-x) x}} \geq -600 \kappa e^{-\frac{\kappa^2}{2}}.
\end{align*}
Take $\delta=0.001$. This implies that
\begin{align*}
    &q_\kappa(x)\\
    &\leq \sum_{y\in T_{\delta}} \delta f(y) + 100\kappa^2 \exp(-\kappa^2/2)\delta^2\\
    &\leq \sum_{y\in T_{\delta}} \delta \left(\erf\left(\frac{1}{\sqrt{2}}\mathcal{S}_u\left(\frac{(\kappa+(1-2x)y)^2}{4x(1-x)}\right)\right)+\erf\left(\frac{1}{\sqrt{2}}\mathcal{S}_u\left(\frac{(\kappa-(1-2x)y)^2}{4x(1-x)}\right)\right)\right)\exp(-\frac{y^2}{2})\\
    &\quad \quad + 100\kappa^2 \exp(-\kappa^2/2)\delta^2\\
    &\leq \sum_{y\in T_{\delta}} \delta  \left(g_u\left(\mathcal{S}_u\left(\frac{(\kappa+(1-2x)y)^2}{4x(1-x)}\right)\right)+g_u\left(\mathcal{S}_u\left(\frac{(\kappa-(1-2x)y)^2}{4x(1-x)}\right)\right)\right)h_u(y)\\
    &\quad \quad + 100\kappa^2 h_u(\kappa)\delta^2:=q_{\kappa,u}(x),
\end{align*}
where $\sum_{y\in T_{\delta}}$ means that we apply the Trapezoidal rule with step size $\delta$. Note that $q_{\kappa,u}(x)$ can be evaluated exactly on rational points. \\

\noindent\textbf{Bound on $\log(x)$.}
We use $\mathcal{G}_l$ and $\mathcal{G}_u$ to denote the upper and lower bound respectively. By Taylor expansion, for $x\in [0.5,1]$,
\begin{align*}
    \log(x)&\geq \sum_{i=1}^{15} \frac{(-1)^{i+1}}{i}(x-1)^i:=\mathcal{G}_{l,0}(x)
\end{align*}
and
\begin{align*}
    \log(x)&\leq \sum_{i=1}^{16} \frac{(-1)^{i+1}}{i}(x-1)^i:=\mathcal{G}_{u,0}(x).
\end{align*}
From this for any $x>1$, let $a$ be the smallest non-negative integer such that $x^{1/2^a}\in [1,2]$. Note that $x^{1/2^a}=\sqrt{x}^{(a)}$, where the notation means that we apply square root $a$ times. Therefore, we have that
\begin{align*}
    &\log(x)=-2^a \log(x^{\frac{1}{-2^a}}) =-2^a \log(1/\sqrt{x}^{(a)})\geq -2^a \mathcal{G}_{u,0}(1/\mathcal{S}_l^{(a)}(x)):=\mathcal{G}_{l}(x)\\
    &\log(x)=-2^a \log(x^{\frac{1}{-2^a}}) =-2^a \log(1/\sqrt{x}^{(a)})\leq -2^a \mathcal{G}_{l,0}(1/\mathcal{S}_u^{(a)}(x)):=\mathcal{G}_{u}(x).
\end{align*}
Similarly, for any $0<x<1$, let $b$ be the smallest non-negative integer such that $x^{-1/2^b}\in [1,2]$. Then
\begin{align*}
    &\log(x)=2^b \log(x^{\frac{1}{2^b}}) =2^b \log(\sqrt{x}^{(b)})\geq 2^b \mathcal{G}_{l,0}(\mathcal{S}_l^{(a)}(x)):=\mathcal{G}_{l}(x)\\
    &\log(x)=2^b \log(x^{\frac{1}{2^b}}) =2^b \log(\sqrt{x}^{(b)})\leq 2^b \mathcal{G}_{u,0}(\mathcal{S}_u^{(a)}(x)):=\mathcal{G}_{u}(x).
\end{align*}
Note that $\mathcal{G}_{u}$ and $\mathcal{G}_{l}$ can be evaluated exactly on rational points.

\subsubsection{Bound on $\beta$}
By Lemma \ref{h:1}, for any $0<\alpha<\alpha_c(\kappa)$, we know that the second derivative satisfies
\begin{align*}
    F''_{\kappa,\alpha}\left(\frac{1}{2}\right)<0.
\end{align*}
This is equivalent to
\begin{align}
    \frac{\pi P_\kappa^2}{\alpha2 \kappa^2 \exp(-\kappa^2)}>1. \label{eq:beta}
\end{align}
Note that
\begin{align*}
    \beta^2=\frac{\alpha}{4}(1-\mu_{2,\kappa})=\frac{\alpha}{4}\frac{2\exp(-\kappa^2)\kappa^2}{\pi P_\kappa^2}.
\end{align*}
As $\beta<0$, it is easy to see that $\beta>-1/2$ is equivalent to \eqref{eq:beta}. Therefore, $-1/2<\beta<0$.

\printbibliography
\end{document}